\documentclass[12pt]{amsart}
\usepackage{amssymb,bbold}
\usepackage{euscript}
\usepackage[hypertex]{hyperref}
\usepackage[all]{xy}
\usepackage[pagewise]{lineno}
\catcode`\@=11
\def\@evenfoot{\rule{0pt}{20pt}[\today] \hfill}
\def\@oddfoot{\rule{0pt}{20pt}\hfill [\today]}
\catcode`\@=13
\allowdisplaybreaks

\def\today{September 30, 2013}

\newtheorem{theorem}{Theorem}[section]

\newtheorem{proposition}[theorem]{Proposition}

\newtheorem{corollary}[theorem]{Corollary}
\theoremstyle{definition}
\newtheorem{definition}[theorem]{Definition}
\newtheorem{example}[theorem]{Example}
\newtheorem{remark}[theorem]{Remark}
\newtheorem{exercise}[theorem]{Exercise}

\newtheorem*{Conjecture}{Conjecture}
\newtheorem*{Corollary}{Corollary}
\newtheorem*{Definition}{Definition}
\newtheorem*{Example}{Example}
\newtheorem*{Remark}{Remark}
\newtheorem*{Exercise}{Exercise}
\newtheorem*{Problem}{Problem}

\def\Broucek{\hbox{$\Freec_{\hskip -.15em\calP}A$}}
\def\Rafik{\hbox{$\Free_{\hskip -.15em \calP}A$}}
\def\eu{e^{(1)}}\def\Liek{\hbox{${\mathbb L \hskip -.25em}^k$}}
\def\Ass{{\mathcal A{\it ss\/}}}\def\Com{{\mathcal C{\it om\/}}}
\def\Lieop{\mathcal L{\it ie\/}}\def\calC{{\mathcal C}}
\def\calP{{\EuScript P}}\def\Free{{\mathbb F}}\def\Freec{{\mathbb F}^c}
\def\sign#1{(-1)^{#1}}\def\sgn#1{(-1)^\varepsilon}\def\Sym{{\mathbb S}}
\def\Ten{{\mathbb T}}\def\Tr{{\EuScript T}}\def\calA{{\mathcal A}}
\def\Ainfty{{$A_\infty$}}\def\Tc{{\mathbb T}^c}\def\Aut{{\rm Aut}}
\def\Linfty{{$L_\infty$}}\def\Sc{{\mathbb S}^c}\def\calL{{\mathcal L}}
\def\uAut{{\underline{\Aut}}(A)}\def\Boj{B\"orjeson}\def\zn#1{{(-1)^{#1}}}
\def\Algsdelta{{\tt Algs}^\Delta}\def\Cinfty{{$C_\infty$}}\def\NIC#1{\relax}
\def\Vect{{\tt Vect}}\def\Lie{\mathbb L}
\def\Liec{\hbox{${\mathbb L \hskip -.25em}^c$}}
\def\bfk{{\mathbb k}}\def\id{1\!\!1}\def\ot{\otimes}\def\Span{{\rm Span}}
\def\otexp#1#2{{#1}^{\otimes #2}}\def\epi{ \twoheadrightarrow}\def\nic{\relax}
\def\Rada#1#2#3{#1_{#2},\dots,#1_{#3}}\def\rada#1#2{{#1},\ldots,{#2}}
\def\cases#1#2#3#4{
                  \left\{
                         \begin{array}{ll}
                           #1,\ &\mbox{#2}
                           \\
                           #3,\ &\mbox{#4}
                          \end{array}
                   \right.
}

\def\tricases#1#2#3#4#5#6{
                  \left\{
                         \begin{array}{rl}
                           #1,\ &\mbox{#2}
                           \\
                           #3,\ &\mbox{#4}
                           \\
                           #5,\ &\mbox{#6}
                          \end{array}
                   \right.
}
\def\ctyricases#1#2#3#4#5#6#7#8{
                  \left\{
                         \begin{array}{rl}
                           #1,\ &\mbox{#2}
                           \\
                           #3,\ &\mbox{#4}
                           \\
                           #5,\ &\mbox{#6}
                           \\
                           #7,\ &\mbox{#8}
                          \end{array}
                   \right.
}
\def\Nat{{\EuScript Nat}}\def\cNat{c{\EuScript Nat}}\def\Fr{\EuScript Fr}
\def\uNat{\underline{\EuScript Nat}}
\def\lachtanek#1#2#3#4#5{
{
\unitlength=.12pt
\begin{picture}(160.00,140.00)(-10.00,50)
\thicklines
\put(120.00,-5){\makebox(0.00,0.00)[t]{\scriptsize $#5$}}
\put(0.00,-5){\makebox(0.00,0.00)[t]{\scriptsize $#4$}}
\put(90.00,40.00){\makebox(0.00,0.00){$#3$}}
\put(30.00,40.00){\makebox(0.00,0.00){$#2$}}
\put(60.00,110.00){\makebox(0.00,0.00){$#1$}}
\put(60.00,70.00){\qbezier(0,0)(30,-30)(60,-60)}
\put(60.00,70.00){\qbezier(0,0)(-30,-30)(-60,-60)}
\put(60.00,140.00){\line(0,-1){70.00}}
\end{picture}}
}

\title{On the origin of higher braces and higher-order derivations}

\author{Martin Markl}

\thanks{The author was supported by the Eduard \v Cech
  Institute P201/12/G028 and RVO: 67985840.}


\catcode`\@=11
\address{Mathematical Institute of the Academy, {\v Z}itn{\'a} 25,
         115 67 Prague 1, The Czech Republic}
\address{MFF UK, 186 75 Sokolovsk\'a 83, Prague 8, The Czech Republic}
\email{markl@math.cas.cz}

\begin{document}
\bibliographystyle{plain}


\begin{abstract}
In {\bf Part~I\/} we show that the classical Koszul
braces~\cite{Koszul}, as well as their non-commutative counterparts
constructed recently in \Boj's~\cite{Boj}, are the twistings of the
trivial \Linfty- (resp.~\Ainfty-) algebra by a specific
automorphism. This gives an astonishingly simple proof of their
properties.  Using the twisting, we construct other surprising
examples of \Ainfty- and \Linfty-braces. We finish Part~1 by
discussing \Cinfty-braces related to Lie~algebras.

In {\bf Part~2\/} we prove that in fact {\em all\/} natural braces are
the twistings by {\em unique\/} automorphisms. We also show that there is
precisely one hierarchy of braces that leads to a~sensible notion of
higher-order derivations. Thus, the notion of higher-order derivations
is independent of human choices. The results of the second part follow
from the acyclicity of a certain space of natural operations.
\end{abstract}

\maketitle

\tableofcontents
\baselineskip 16pt plus 1pt minus 1 pt
\noindent 
{\bf Advice for the reader.}
The article can be read in four possible ways:
\begin{enumerate}
\item 
as a tool for checking that the classical braces indeed
    form \Linfty- resp.~\Ainfty-algebras,

\item
as a machine producing explicit examples of \Linfty-,
\Ainfty- and possibly also other types of strongly homotopy algebras,

\item
as a justification that the higher-order derivations are
    God-given, not human, inventions existing since the beginning of time, or

\item
as a vanilla version of~\cite{batanin-markl}.
\end{enumerate}

The reader wanting only (1) and (2) may read Part~1 and skip the
rest. Item~(3) explains why higher-order derivations of commutative
associative algebras appear e.g.~in the interpretation of the algebraic
structure of the combined conformal field theory of matter and ghosts
given in~\cite{markl:la}. It is natural to expect that higher-order
derivations of associative algebras based on \Boj's braces would play
a similar r\^ole for open strings.

\noindent 
{\bf Plan of the paper.}
{\it Section \ref{sec:exampl-place-intr}\/} contains several
examples of braces, including the classical Koszul \Linfty-hierarchy and \Boj's
\Ainfty-braces. We demonstrate various properties which
the braces may posses, in particular those leading to a sensible
definition of higher-order derivations. 

In {\em Section \ref{sec:constr-high-brack}\/} we show how to generate
braces by the twisting and interpret all
examples in Section \ref{sec:exampl-place-intr} as emerging this
way. This offers a very simple verification that
they indeed form \Linfty- resp.~\Ainfty-structures. We close this
section by discussing possible
generalizations to Lie and other types of algebras.

In {\em Section \ref{sec:naturality}\/} we analyze natural operations
and prove that they form an acyclic space. The main results are
Propositions~\ref{pojedu_Jarce_zalit_kyticky}
and~\ref{Jarka_je_na_chalupe_s_M1}. The material of this section is a
baby version of the analysis of the Hochschild cochains in connection to
Deligne's conjecture as given in~\cite{batanin-markl}.

{\em Section~\ref{sec:main-results-5}\/} formulates the consequences
of Section \ref{sec:naturality}.
Theorems~\ref{sec:main-results} and~\ref{sec:main-results_bis} state
that all natural braces are the twistings by unique automorphisms, Corollaries
\ref{sec:main-results-6} and~\ref{sec:main-results-6_bis} then
describe the moduli space of all natural braces.   
By Theorems \ref{sec:main-results-1} and \ref{sec:main-results-1_bis},
\Boj's resp.~Koszul braces are the unique ones leading to a
meaningful notion of higher-order derivations of associative
resp.~commutative associative algebras.

\noindent 
{\bf Conventions.}
If not stated otherwise, all algebraic objects will be considered
over a fixed field $\bfk$ of characteristic zero.  
The symbol $\otimes$ will denote the
tensor product over $\bfk$ and $\Span(S)$ the $\bfk$-vector space
spanned by a set $S$. We will
denote by $\id_X$ or simply by $\id$ when $X$ is understood, the identity
endomorphism of an object $X$ (set, vector space, algebra, \&c.).
We will usually write the product of elements $a$ and $b$ of an
associative algebra as $a\cdot b$ or
simply as $ab$.

A degree of a graded object will be denoted by $|w|$ though we will
sometimes omit the vertical bars and write e.g.~$(-1)^{a+b}$ instead 
of~$(-1)^{|a|+|b|}$ to save the space. 
For a permutation $\sigma\in\Sigma_k$ and graded variables $\Rada
  w1k$, the
\emph{Koszul sign} $\varepsilon(\sigma;\Rada   w1k) 
\in \{-1,+1\}$ is defined by the equation
\[
w_1\wedge\ldots\wedge
w_k=  \varepsilon(\sigma;\Rada   w1k) \cdot  w_{\sigma(1)}\wedge\ldots\wedge
w_{\sigma(k)}
\]
which has to be satisfied in the free graded commutative associative
algebra ${\mathbb S}(\Rada w1k)$ generated by $\Rada w1k$. We usually
write  $\varepsilon(\sigma)$ instead of
$\varepsilon(\sigma;\Rada   w1k) $ when the meaning of $\Rada   w1k$
is clear from the context.

Given integers $a,b \geq 0$, an $(a,b)$-{\em unshuffle\/} is a 
permutation $\sigma \in \Sigma_{a+b}$ satisfying
\[
\mbox {
$\sigma(1)<\cdots<\sigma(a)$ \ and \ $\sigma(a+1)< \cdots < \sigma(a+b)$.}
\]
By {\em braces\/} we mean the structure operations of a
strongly homotopy algebra.

\noindent 
{\bf Acknowledgment.} I would like to express my thanks to Maria
Ronco for Remark~\ref{Ceka_mne_cesta_proti_strasnemu_vetru}.

\part{Examples and constructions}

\section{Examples in place of introduction}
\label{sec:exampl-place-intr}

We start by recalling a construction attributed to
Koszul~\cite{Koszul} and sometimes referred to as the Koszul hierarchy, see
also~\cite{akman:preprint97,akman-ionescu,bering-damgaard-alfaro:preprint,bering:CMP07,voronov03:higher}. It is used to define
higher order derivations of commutative associative algebras, see
\S\ref{sec:high-order-deriv} below; they play a
substantial r\^ole for instance in the BRST approach to closed string
field theory~\cite[Section~4]{markl:la}.

\begin{example}[Classical \Linfty-braces]
\label{dnes_vecer_s_Jarkou_a_Tondou_u_Pakousu} 
Let $A$ be a graded commutative associative algebra with a degree
$+1$ differential $\Delta$ which is, very crucially, {\em not necessarily\/} a
derivation. {\em Koszul braces\/} are linear degree $+1$ maps
$\Phi_k^\Delta : \otexp Ak \to A$, $k \geq 1$, defined
by the formulas
\begin{align*}
\Phi^\Delta_1(a) &= \Delta(a),
\\
\Phi^\Delta_2(a_1,a_2) &= \Delta(a_1a_2)
 - \Delta(a_1)a_2 - \sign{\nic a_1\nic }a_1\Delta(a_2),  
\\
\Phi^\Delta_3(a_1,a_2,a_3) &=\Delta(a_1a_2a_3)
-\Delta(a_1a_2)a_3 
- \sign{\nic a_1\nic (\nic a_2\nic  + \nic a_3\nic)}  \Delta(a_2a_3)a_1 
-\sign{\nic a_3\nic (\nic a_1\nic  + \nic a_2\nic )}\Delta(a_3a_1)a_2 
\\
&\hphantom{=} \hskip .5em + \Delta(a_1)a_2a_3+ 
\sign{\nic a_1\nic (\nic a_2\nic  
+ \nic a_3\nic )}\Delta(a_2)a_3a_1+\sign{\nic a_3\nic (\nic a_1\nic  +
  \nic a_2\nic )} \Delta(a_3)a_1a_2,
\\
& \hskip .7em \vdots    
\\
\Phi^\Delta_k(\Rada a1k) &= 
\sum_{1 \leq i \leq k} \sign{k-i} \sum_\sigma \varepsilon(\sigma)\Delta(a_{\sigma(1)}
  \cdots a_{\sigma(i)})a_{\sigma(i+1)} \cdots a_{\sigma(k)},
\end{align*}
for $a,a_1,a_2,a_3, \ldots\in A$.  The summation in the last line runs
over all $(i,k-i)$-unshuffles $\sigma$ and $\varepsilon(\sigma) =
\varepsilon(\sigma;\Rada a1k)$ is the Koszul sign. As proved for
instance in~\cite{bering-damgaard-alfaro:preprint}, these braces form an
\Linfty-algebra\footnote{We will give a short and elegant proof of
this fact in Example~\ref{Dojde_zitra_Jaruska_na_parnik?};
\Linfty-algebras are recalled in
Definition~\ref{sec:ainfty-algebras-bis}.} 
and have
moreover the property that
\begin{equation}
\label{eq:10bis}
\mbox {
if $\Phi^\Delta_k = 0$  identically on $\otexp Ak$ then $\Phi^\Delta_{k+1} = 0$
identically on $\otexp A{k+1}$.}
\end{equation}
We call braces with this property {\em hereditary\/}.
\end{example}

\subsection{Higher order derivations}
\label{sec:high-order-deriv}

Let $A$ be a graded commutative associative algebra with a
differential $\Delta$ as in
Example~\ref{dnes_vecer_s_Jarkou_a_Tondou_u_Pakousu}.  
One says \cite{akman-ionescu} that $\Delta$ is an {\em order $r$
derivation\/} if $\Phi^\Delta_{r+1} = 0$. Clearly, being an order $1$
derivation is the same as being a derivation in the usual sense. 
It is almost clear that an order $r$-derivation is determined by its
values on the products $x_1 \cdots x_s$, $s \leq r$, of generators of
$A$; an explicit formula is given in~\cite[Proposition~3.4]{markl:la}.

One may ask whether higher-order derivations are `God-given,'
i.e.~whether the braces that define it are unique.  Let us try to
find out which properties the braces leading to a sensible notion of
higher-order derivations should satisfy.  First of all, they must be
`natural' in that they use only the data that are
available for any graded associative commutative algebra with a
differential. The exact meaning of naturality is analyzed in
Section~\ref{sec:naturality}.

Given \Linfty-braces
$(A,\Delta,l^\Delta_2,l^\Delta_3,\ldots)$, we may call $\Delta$ an order $r$
$l$-derivation if and only if $l^\Delta_{r+1} = 0$. 
It is clear from the axioms for \Linfty-algebras recalled in
\S\ref{sec:comm-algebr-linfty} that, for arbitrary scalars $\alpha,\beta \in
\bfk$, the object
\begin{equation}
\label{eq:10}
(A,\alpha\Delta,\alpha\beta\, l^\Delta_2,\alpha\beta^2\, l^\Delta_3,
\alpha\beta^3\, l^\Delta_4,\ldots)
\end{equation} 
is an \Linfty-algebra as well.  The first property we want is that an
order $1$ $l$-derivation is an ordinary derivation.  This means that,
after a suitable renormalization~(\ref{eq:10}),
\begin{subequations}
\begin{equation}
\label{eq:12}
l^\Delta_2(a_1,a_2) = \Delta(a_1a_2)
 - \Delta(a_1)a_2 - \sign{|a_1|}a_1\Delta(a_2), \ \mbox { for each } a_1,a_2 \in A.
\end{equation}
We also certainly want that an order  $r$
$l$-derivation is also an order $r+1$
$l$-derivation, that~is:
\begin{equation}
\label{eq:15}
\mbox {the braces $l^\Delta_2,l^\Delta_3,l^\Delta_4, \ldots$ are hereditary.} 
\end{equation}
The following example however shows that
conditions~(\ref{eq:12})--(\ref{eq:15}) still do not determine the
braces uniquely.

\begin{example}[Hereditary exotic \Linfty-braces]
\label{Preletim_Vivata?}
Here $A$ is a graded commutative associative algebra with a degree
$+1$ differential $\Delta$ as in
Example~\ref{dnes_vecer_s_Jarkou_a_Tondou_u_Pakousu}. 
{}For $a,a_1,a_2,a_3,\ldots \in A$ define
\begin{align*}
h^\Delta_1(a) &= \Delta(a),
\\
h^\Delta_2(a_1,a_2) &= \Delta(a_1a_2)
 - \Delta(a_1)a_2 - \sign{\nic a_1\nic }a_1\Delta(a_2),  
\\
h^\Delta_2(a_1,a_2,a_3) &= 2\Delta(a_1)a_2a_3 +
  \sign{\nic a_1\nic (\nic a_2\nic  + \nic a_3\nic )} 2 \Delta(a_2)a_3a_1 
+ \sign{\nic a_3\nic (\nic a_1\nic  + \nic a_2\nic )}2 \Delta(a_3)a_1a_2
\\
&\hphantom{=} \hskip .5em - \Delta(a_1a_2)a_3   
- \sign{\nic a_1\nic (\nic a_2\nic  + \nic a_3\nic )}\Delta(a_2a_3)a_1
- \sign{\nic a_3\nic (\nic a_1\nic  + \nic a_2\nic )}\Delta(a_3a_1)a_2 , 
\\
&\hskip .7em \vdots
\\    
h^\Delta_k(\Rada a1k) &= \zn{k+1} (k\!-\!1)! \sum_\tau  \varepsilon(\tau)  
\Delta(a_{\tau(1)})a_{\tau(2)} \cdots a_{\tau(k)}
\\
&\hskip 3.7em + \zn k
(k\!-\!2)! \sum_\sigma \varepsilon(\sigma)  
\Delta(a_{\sigma(1)}a_{\sigma(2)})a_{\sigma(3)} \cdots a_{\sigma(k)}
\end{align*}
where $\tau$ runs over all $(1,k\!-\!1)$-unshuffles 
and $\sigma$ over all $(2,k\!-\!2)$-unshuffles. 
It is easy to verify that the above braces satisfy the induction
\[
h^\Delta_{k+1}(\Rada a1{k+1}) = - \sum_{\sigma} \varepsilon(\sigma) 
h^\Delta_k(\Rada a1{\sigma(k)})  a_{\sigma(k+1)}
\]
with $\sigma$ running over all $(k,1)$-unshuffles.
This implies that they are hereditary.
\end{example}

Next, we want the {\em recursivity\/} of higher-order derivations, 
by which we mean that an order $r$
$l$-derivation is determined by its values on the products
of $\leq r$ generators. Moreover, the notion of higher-order
derivations and therefore the braces as well must be defined over an
arbitrary ring. The recursivity is thus equivalent to:
\begin{equation}
\label{eq:17}
\begin{array}{cc}
\mbox {The braces are defined over the ring ${\mathbb Z}$ of integers
and the coefficient $C_k$}
\\ \mbox{at the term $\Delta(a_1\cdots a_k)$
\rule{0em}{1.1em}in $l^\Delta_k(a_1,\ldots, a_k)$ 
is either $+ 1$ or $-1$ for any $k \geq 1$.}
\end{array}
\end{equation}
\end{subequations}
For, if $p := C_k \not\in \{-1,1\}$ for some $k$, then higher-order
$l$-derivations will not be recursive over the ring ${\mathbb
Z}/p{\mathbb Z}$ of integers modulo $p$. For the braces in
Example~\ref{Preletim_Vivata?}, $C_k = 0$ for all $k \geq 3$, so they
do not satisfy~(\ref{eq:17}).

It will follow from Theorem~\ref{sec:main-results-1_bis} that
assumptions~(\ref{eq:12})--(\ref{eq:17}) already imply that $l^\Delta_k =
\Phi^\Delta_k$ for each $k \geq 1$. Let us
start our discussion of the non-commutative
case by recalling one construction from a recent preprint~\cite{Boj} of \Boj.

\begin{example}[B\"orjeson's \Ainfty-braces] 
\label{Jarka_mi_vcera_volala}
Given a graded associative (not necessarily commutative) algebra $A$ with a
derivation $\Delta$, define for $a,a_1,a_2,a_3,\ldots \in A$,
\begin{align*}
b^\Delta_1(a) &= \Delta(a),
\\
b^\Delta_2(a_1,a_2) &= \Delta(a_1a_2) - \Delta(a_1)a_2 - \zn {a_1} 
a_1\Delta(a_2),  
\\
b^\Delta_3(a_1,a_2,a_3) &=\Delta(a_1a_2a_3) -\Delta(a_1a_2)a_3 -
\zn {a_1}  a_1\Delta(a_2a_3)
+\zn {a_1}  a_1\Delta(a_2)a_3,
\\
b^\Delta_4(a_1,a_2,a_3,a_4) &=\Delta(a_1a_2a_3a_4)\! - \! \Delta(a_1a_2a_3)a_4
\!-\!\zn {a_1}  a_1\Delta(a_2a_3a_4)\! +\!\zn {a_1}  a_1\Delta(a_2a_3)a_4,
\\
&\hskip .6em \vdots
\\
b^\Delta_k(\Rada a1k) &= 
\Delta(a_1\cdots a_k) - \Delta(a_1\cdots a_{k-1})a_k 
\\
&\hphantom{=} \hskip .5em-
\zn{a_1}a_1\Delta(a_2\cdots a_{k}) + \zn{a_1}a_1 \Delta(a_2\cdots a_{k-1})a_k.
\end{align*}
As proved in~\cite{Boj}, these braces form an
\Ainfty-algebra\footnote{A simple proof of this fact is provided by
Example~\ref{Jaruska_mozna_prijede_ve_ctvrtek_do_Prahy};
\Ainfty-algebras are recalled in Definition~\ref{sec:ainfty-algebras}.} and are hereditary.
\end{example}

It is obvious that \Boj's braces satisfy
assumptions~(\ref{eq:12})--(\ref{eq:17}), so they lead to a sensible
notion of higher-order derivations of graded associative
(non-commutative) algebras. By Theorem~\ref{sec:main-results-1}, they
are the only \Ainfty-braces with these properties.

\begin{example}[Non-recursive \Ainfty-braces]
\label{Pisu_v_Koline-vcera-jsem-prijel-na-kole!}
\label{koupili-jsme-si-pomucku}
The braces below lead to recursive
higher-order derivations over ${\mathbb Z}$ but not over
${\mathbb Z}/5{\mathbb Z}$, the integers modulo $5$.
They are, up to the obvious Koszul signs, given by
\begin{align*}
e^\Delta_1(a) &= \Delta(a),
\\
e^\Delta_2(a_1,a_2)&= \Delta(a_1a_2) - \Delta(a_1)a_2 - a_1\Delta(a_2) 
\\
e^\Delta_3(a_1,a_2,a_3) &= 2\Delta(a_1a_2a_3) - \Delta(a_1)a_2a_3 - a_1a_2\Delta(a_3) - \Delta(a_1a_2)a_3 -
a_1\Delta(a_2a_3),
\\
e^\Delta_4(a_1,a_2,a_3,a_4) &= 5 \Delta(a_1a_2a_3a_4) -
\Delta(a_1a_2)a_3a_4 - a_1a_2\Delta(a_3a_4) 
\\
& \ \ \ \  - 2\big(\Delta(a_1)a_2a_3a_4 + a_1a_2a_3\Delta(a_4) +
\Delta(a_1a_2a_3)a_4 + a_1\Delta(a_2a_3a_4)\big)
\\
&\ \ \vdots
\\
e^\Delta_k(\Rada a1k) &=
\alpha_k \Delta(a_1\cdots a_k)
\\
& \ \ \ \ - \sum_{1\leq u \leq k-1} \alpha_u \alpha_{k-u}
\big(\Delta(a_1\cdots a_u) a_{u+1} \cdots a_k + a_1 \cdots a_u
\Delta(a_{u+1} \cdots a_k)  \big),
\end{align*}
where
\begin{equation}
\label{Jaruska_ma_zitra_krasne_54te_narozeniny!}
\alpha_1 := 1 \ \mbox { and } \ 
\alpha_k := \frac 1{k-1} \binom {2k-2}k \ \mbox { for }  \ k \geq 2.
\end{equation}  
\end{example}

\begin{example}[Hereditary non-recursive \Ainfty-braces]
\label{pojedu-proti-vetru}
We define braces
satisfying~(\ref{eq:12}),~(\ref{eq:15}) but not~(\ref{eq:17}).
Namely, for elements $a,a_1,a_2,a_3,\ldots$ of a graded associative algebra $A$ with a differential
$\Delta$ we put
\begin{align*}
e^\Delta_1(a) &= \Delta(a),
\\
e^\Delta_2(a_1,a_2) &= \Delta(a_1a_2) - \Delta(a_1)a_2 - a_1\Delta(a_2),  
\\
e^\Delta_3(a_1,a_2,a_3) &= -\Delta(a_1a_2)a_3 - a_1\Delta(a_2a_3)
+\Delta(a_1)a_2a_3 + 2a_1\Delta(a_2)a_3 + a_1a_2\Delta(a_3), 
\\
e^\Delta_4(a_1,a_2,a_3,a_4) &= \Delta(a_1a_2)a_3a_4 +2a_1\Delta(a_2a_3)a_4 +
a_1a_2\Delta(a_3a_4)
\\
&\hphantom{=} -
\Delta(a_1)a_2a_3a_4  -3a_1\Delta(a_2)a_3a_4 -3a_1a_2\Delta(a_3)a_4
-a_1a_2a_3\Delta(a_4),
\\
&\hphantom{=}\vdots    
\\
e^\Delta_k(\Rada a1k) &= (-1)^k
\sum_{0 \leq i \leq k -2} {\textstyle\binom {k-2}i} \  a_1 \cdots
\Delta(a_{i+1}a_{i+2}) \cdots a_k
\\
&\hskip 5em - (-1)^k\sum_{0 \leq i \leq k -1} {\textstyle\binom {k-1}i} \
 a_1 \cdots
\Delta(a_{i+1}) \cdots a_k.
\end{align*}
We omitted for clarity the obvious Koszul signs.
It is easy to verify the inductive formula
\[
e^\Delta_{k+1}(\Rada a1{k+1}) = 
-\zn {a_1} a_1e^\Delta_k(\Rada a2{k+1}) - e^\Delta_k(\Rada
a1k)a_{k+1}, \ k \geq 1,  
\]
which implies that they are hereditary. On the other hand,
$e^\Delta_k(\Rada a1k)$ does not contain the term $\Delta(a_1\cdots a_k)$,
so the coefficients $C_k$ in~(\ref{eq:17}) are $0$ for all $k \geq
2$.    
\end{example}

Hereditarity is a very fine property; `randomly chosen' braces will
not be hereditary. A systematic method of producing non-hereditary
braces, based surprisingly on a rather deep
Proposition~\ref{sec:main-results-3}, is described in
Example~\ref{zitra_prijede_Ronco}.

\section{Constructions of higher braces}
\label{sec:constr-high-brack}

\subsection{Non-commutative algebras and \Ainfty-braces}
\label{zase_mi_vynechavalo_srdce}
Recall that an \Ainfty-algebra consist of a graded vector space $V$
together with linear operations $\mu_k : \otexp Vk \to V$, $k \geq 1$, such that
$\deg(\mu_k) = 2-k$, satisfying a system of axioms that say that
$\mu_1$ is a differential, $\mu_2$ is associative up to the homotopy
$\mu_3$, \&c, see e.g. \cite{stasheff:TAMS63}.

It will be useful in the context of this paper to transfer the
operations $\mu_k : \otexp Vk \to V$ to the desuspension $A :=\
\downarrow\! V$, i.e.~to define new operations $m_k : \otexp Ak \to A$
by the commutativity of the diagram
\[
\xymatrix{
\otexp Ak \ar[r]^{m_k}  &  A
\\
\otexp Vk\ar[u]^{\otexp \downarrow k} \ar[r]^{\mu_k} &\  V, \ar[u]_{\downarrow}
}
\]
where $\downarrow : V \to\ \downarrow\! V = A$ is the desuspension
map. All $m_k$'s then are of degree $+1$ and they satisfy the axioms
\begin{equation}
\label{eq:8}
\sum_{k+l = n+1}\sum_{1\leq i \leq k}
m_k(\id_A^{\ot {i-1}} \ot m_l \ot \id_A^{\ot {k-i}}) = 0, \ \mbox { for
  each $n \geq 1$}.
\end{equation}
We will use this version of \Ainfty-algebras throughout the paper:

\begin{definition}
\label{sec:ainfty-algebras}
An {\em \Ainfty-algebra\/} is a structure $\calA =
(A,m_1,m_2,m_3,\ldots)$ consisting of a graded vector space $A$ and
degree $+1$ linear maps $m_k : \otexp Ak\to A$, $k \geq 1$,
satisfying~(\ref{eq:8}). 
\end{definition}

Let $\Tc A$ be the coalgebra whose underlying space is the tensor
algebra $\Ten A := \bigoplus_{n \geq 1} \otexp An$ and the diagonal
(comultiplication) is the de-concatenation.  It turns out that
$\Tc A$ is a cofree {\em conilpotent\/} coassociative coalgebra
cogenerated by $A$, see e.g.~\cite[\S
II.3.7]{markl-shnider-stasheff:book}.\footnote{A general misconception
is that $\Tc A$ is cofree in the category of {\em all\/} coassociative
coalgebras.}  Its cofreeness implies that each coderivation $\vartheta$
of $\Tc A$ is given by its components $\vartheta_k : \otexp Ak \to A$, $k \geq 1$,
defined by $\vartheta_k := \pi \circ
\vartheta \circ \iota_k$, where $\pi: \Tc A \epi A$ is the
projection and $\iota_k : \otexp Ak \hookrightarrow \Tc A$ the
inclusion. We write $\vartheta =
(\vartheta_1,\vartheta_2,\vartheta_3,\ldots)$.

Let $m := (m_1,m_2,m_3,\ldots)$ be a degree $1$ coderivation of $\Tc
A$ determined by the linear maps $m_k$ as in Definition
\ref{sec:ainfty-algebras}. It is well-known that
axiom~(\ref{eq:8}) is equivalent to $m$ being a differential,
i.e.~to a single equation $m^2 = 0$.  Therefore equivalently,
an \Ainfty-algebra is a pair $(A,m)$ consisting of a graded vector space
$A$ and a degree $+1$ coderivation $m$ of $\Tc A$ which squares to zero.

In homological algebra one usually considers \Ainfty-algebras
$(A,m_1,m_2,m_3,\ldots)$ as objects living in the category of
differential graded (dg) vector spaces, the linear operation
(differential) $m_1$ being part of its underlying dg-vector space, not a
structure operation. For this reason we call an \Ainfty-algebra
with $m_k = 0$ for $k\geq 2$ a {\em trivial\/} \Ainfty-algebra.

\begin{example}[Trivial \Ainfty-algebra]
\label{Vcera_mi_Jarunka_volala_kdyz_jsem_byl_v_Libni}
Let $\Delta: A \to A$ be a degree $+1$ differential on a graded vector space
$A$. It is clear that $\calA_\Delta := (A,\Delta,0,0,\ldots)$ 
is an \Ainfty-algebra. The differential $\Delta$ extends to a linear
coderivation $(\Delta,0,0,\ldots)$ of $\Tc A$.
\end{example}

As coderivations, by the universal property of $\Tc A$ each
endomorphism $\phi : \Tc A \to \Tc A$ is determined by its components
$\phi_k : \otexp Ak \to A$, $k \geq 1$, defined by $\phi_k := \pi
\circ \phi \circ \iota_n$,  We will
write $\phi = (\phi_1,\phi_2,\phi_3,\ldots)$. 
The sequence
$(\id_A,0,0,\ldots)$ represents the identity automorphism. The composition
$\psi\phi$ of $\phi$ with another endomorphism $\psi =
(\psi_1,\psi_2,\psi_3,\ldots)$  has components
\begin{equation}
\label{opet_jsem_podlehl}
(\psi \phi)_k = \sum_{r \geq 1} \sum_{i_1 + \cdots + i_r = k} 
\psi_r(\phi_{i_1} \ot \cdots \ot \phi_{i_r}).
\end{equation}
It is well-known that $\phi : \Tc A \to \Tc A$ is an automorphism
(i.e.~invertible endomorphism) if
and only if $\phi_1 : A \to A$ is invertible. We call $\phi$ {\em
linear\/} if $\phi_k= 0$ for $k \geq 2$. Let us recall

\begin{Definition}
Two \Ainfty-algebras $\calA' = (A',m')$ and $\calA'' = (A',m'')$ are
{\em isomorphic\/} if there exists an automorphism $\phi : \Tc A' \to
\Tc A''$ such that $\phi m' = m'' \phi$.\footnote{Sometimes one says
that $\calA'$ and $\calA''$ are {\em weakly\/} isomorphic.} They are
{\em strictly\/} isomorphic if there exist a {\em linear\/} $\phi$ as
above.
\end{Definition}

Assume that we are given an \Ainfty-algebra $\calA = (A,m)$ and an 
automorphism $\phi : \Tc A \to \Tc A$. Then clearly $(A,\phi^{-1} m
\phi)$ is an \Ainfty-algebra isomorphic to $\calA$.

\begin{definition}
\label{V_nedeli_se_Jarunka_vrati}
In the situation above, we denote $m^\phi :=\phi^{-1} m \phi$
and call the \Ainfty-algebra $\calA^\phi :=
(A,m^\phi)$  the {\em
twisting\/} of the \Ainfty-algebra $\calA = (A,m)$ by the automorphism $\phi$.
\end{definition}

The components of the twisted coderivation $m^\phi := \phi^{-1} m
\phi$ can be expressed explicitly as 
\[
m^\phi_k =
\sum_{r,u \geq 1}\sum_{1\leq j \leq r}\sum_{i_1 + \cdots + i_{r+u-1} = k} 
(\phi^{-1})_r(\otexp {\id}{j-1} \ot m_u \ot \otexp {\id}{r-j})
(\phi_{i_1} \ot \cdots \ot \phi_{i_{r+u-1}}).
\]
If $m$ is the linear coderivation $\Delta$ as in Example
\ref{Vcera_mi_Jarunka_volala_kdyz_jsem_byl_v_Libni}, the above formula
simplifies~to
\begin{equation}
\label{musim_napsat_zapis}
\Delta^\phi_k =
\sum_{r \geq 1}\sum_{1\leq j \leq r}\sum_{i_1 + \cdots + i_{r} = k} 
(\phi^{-1})_r(\otexp {\id}{j-1} \ot \Delta \ot \otexp {\id}{r-j})
(\phi_{i_1} \ot \cdots \ot \phi_{i_{r}}).
\end{equation}

\subsection{Explicit formulas}

Let $A$ be graded associative algebra with the product $\mu :
\otexp A2 \to A$. 
Denote by $\uAut$ the group of automorphisms $\phi$ of $\Tc(A)$
of the form 
\[
\phi := (\id_A,f_2\, \mu_2,f_3\, \mu_3,f_4\,\mu_4,\ldots), 
\]
where $\mu_k: \otexp Ak \to A$ is the multiplication $\mu$ iterated
$(k-1)$-times, and $f_k \in \bfk$ are scalars, $k \geq 2$.\footnote{As
  explained in Example~\ref{snad-se-mi-neotoci-vitr}, $\uAut$ is a
  sub-monoid of the monoid of all natural automorphisms.}
Such an
automorphism is clearly determined by its {\em generating series\/} 
\begin{equation}
\label{mam-roztrhane_trenky}
\phi(t) := t + f_2t^2 + f_3 t^3 +\cdots \in \bfk [[t]].
\end{equation}
It is easy to verify using~(\ref{opet_jsem_podlehl}) that the
composition of automorphisms is translated into the composition of their
generating series, i.e.~$(\psi\phi)(t) = \psi\big(\phi(t)\big)$.

Let $\phi\in \uAut$ be the automorphism with the generating
series~(\ref{mam-roztrhane_trenky}) and $\psi(t)
:= \phi^{-1}(t)$ its inverse with the generating series
\[
\psi(t) = t + g_2t^2 + g_3 t^3
+\cdots \in \bfk [[t]].
\]
Denote by
\[
\psi'(t) := 1 + g_2t + g_3 
t^2 + g_4t^3\cdots \in \bfk [[t]]
\]
the \underline{non}commutative derivative of $\psi(t)$. 
It is straightforward to verify  
that~(\ref{musim_napsat_zapis}) gives
\begin{equation}
\label{v_pondeli_mne_snad_Jaruska_zachrani}
\Delta^\phi_k = \sum_{r+s+p = k}
c_{r,s} \cdot f_p \cdot
\mu_{r+s+1}(\otexp \id r \ot \Delta \mu_p \ot \otexp \id s),
\end{equation}
with the coefficients  $c_{r,s} \in \bfk$ defined as
\begin{equation}
\label{lednim_medvidkum_je_vedrobis}
c_{r,s} := \left. \psi'\big(\phi(u) + \phi(v)\big)
\right|_{u^rv^s} ,
\end{equation}
where  $\big |_{u^rv^s}$ denotes the coefficient at
$u^rv^s$ of the corresponding power series in the ring of {
noncommutative\/} polynomials in $u$ and $v$, i.e.~the
{noncommutative\/} Taylor coefficient at $u^rv^s$. 
Explicitly,
\begin{equation}
\label{dnes_ma_byt_40}
c_{r,s} = \ \sum_{k,l \geq 0} g_{k+l+1} \sum_{a_1+\cdots+a_k = r} f_{a_1}
\cdots f_{a_k} \sum_{b_1+\cdots+b_l = s} f_{b_1} \cdots f_{b_l},
\end{equation}
where we put, by definition, $g_1=f_1 := 1$.
Observe that the above sum makes sense even for $r$ or $s$ equaling
$0$ provided we interpret the empty product as $1$.

\begin{Exercise}
If $\Delta$ is a derivation, then $\Delta^\phi = \Delta$, i.e.~$\Delta_k = 0$
in~(\ref{v_pondeli_mne_snad_Jaruska_zachrani}) for all $k \geq 2$.
\end{Exercise}

\begin{example}[B\"orjeson's \Ainfty-braces continued] 
\label{Jaruska_mozna_prijede_ve_ctvrtek_do_Prahy}
We describe the braces constructed in~\cite{Boj}
and recalled in Example~\ref{Jarka_mi_vcera_volala} as a
twisting of the trivial \Ainfty-algebra $\calA_\Delta = 
(A,\Delta,0,0,\ldots)$. We take as
$\phi$ the automorphism with the generating series 
\[
\phi(t) := t + t^2 + t^3 + \cdots = \frac t{1-t},
\]
so that
\[
\psi(t) := \phi^{-1}(t) = t - t^2 + t^3 - \cdots = \frac t{1+t}.
\]
In this case, $g_{k+l+1}$ in~(\ref{dnes_ma_byt_40}) equals $(-1)^{k+l}$,
therefore
\begin{align*}
c_{r,s}
&= \sum_{a_1+\cdots+a_k = r} (-1)^k{f_1} \cdots  f_{a_k}
\sum_{b_1+\cdots+b_l = s} (-1)^l f_{b_1} \cdots  f_{b_l}
\\
&=\left.\big(1 - \psi\phi(u)\big)\big(1-\psi\phi(v)\big)
\right|_{u^rv^s}
= \left.(1-u)(1-v) \right|_{u^rv^s}.
\end{align*}
We conclude that
\[
c_{r,s} = \tricases {1}{if $(r,s) \in  \big\{(0,0),(1,1)\big\}$,}
{-1}{if $(r,s) \in  \big\{(0,1),(1,0)\big\}$, and}
0{in the remaining cases.}
\]
Since $f_p =1$, formula~(\ref{v_pondeli_mne_snad_Jaruska_zachrani}) 
obviously gives \Boj' braces, i.e.~$b_k^\Delta =  \Delta^\phi_k$ for
each $k \geq 1$.
\end{example}

\begin{example}[Non-recursive \Ainfty-braces continued]
\label{sec:explicit-formulas}
The braces described in
Example~\ref{Pisu_v_Koline-vcera-jsem-prijel-na-kole!} are the result
of the twisting by the endomorphism $\phi$ with the generating function
\[
\phi(t) := \frac{1-\sqrt{1-4t}}2 = 
t + \sum_{k \geq 2}  \frac {t^k}{k-1} \binom {2k-2}k =
t+t^2+2t^3+5t^4+
14t^5 +\cdots, 
\]
whose inverse equals $\psi(t) = t-t^2$.
Observe that $\psi'(t) = 1-t$, therefore,
in~(\ref{lednim_medvidkum_je_vedrobis})
\[
c_{r,s} = \left. \big(1 - \phi(u) - \phi(v)\big)
\right|_{u^rv^s}
=
\tricases{-(\alpha_r+\alpha_s)}{if $r=0$ or $s=0$ but $(r,s) \not= (0,0)$,}
1{if $(r,s) = (0,0)$, and}0{in the remaining cases,}
\]
where $\alpha_i$'s are the Taylor coefficients of $\phi(t)$. It is simple to
verify that formula~(\ref{v_pondeli_mne_snad_Jaruska_zachrani}) leads to
the braces of Example~\ref{Pisu_v_Koline-vcera-jsem-prijel-na-kole!}.
\end{example}

\begin{Example}[Hereditary non-recursive \Ainfty-braces continued]
The braces in Example~\ref{pojedu-proti-vetru} are generated by the
automorphism with the generating series $\phi(t) = t+ t^2$ whose
inverse is
\[
\psi(t) := \frac{\sqrt{1+4t}-1}2 = 
t - \sum_{k \geq 2}  \frac {(-t)^k}{k-1} \binom {2k-2}k =
t-t^2+2t^3-5t^4+14t^5 - \cdots.
\]
We leave as an exercise to perform the calculation. 
Notice that $\phi(t)$ and $\psi(t)$ are related with those from
Example~\ref{sec:explicit-formulas} via the transformation
\[
\phi(t) \mapsto -\psi(-t),\ \psi(t) \mapsto -\phi(-t).
\]
\end{Example}

\subsection{Commutative algebras and \Linfty-braces}
\label{sec:comm-algebr-linfty}

Lie counterparts of \Ainfty-algebras are \Linfty-algebras.  An
\Linfty-algebra is a graded vector space $L$ with linear operations
$\ell_k : \otexp Lk \to V$, $k \geq 1$, $\deg(\ell_k) = 2-k$, that are
graded antisymmetric and satisfy axioms that say that $\ell_1$ is a
differential, $\ell_2$ fulfills the Jacobi identity up to the homotopy
$\ell_3$, \&c, see e.g.\
\cite{lada-markl:CommAlg95,lada-stasheff:IJTP93}.

As for \Ainfty-algebras, we will use the version 
transferred  the desuspension $A :=\ \downarrow\! L$. 
The transferred structure operations $l_k$'s have degree 
$+1$, are graded symmetric, and satisfy, for each $\Rada a1n
\in A$, the `master identity'
\begin{equation}
\label{eq:8bis}
\sum_{i+j=n+1}\sum_\sigma
\varepsilon(\sigma)l_j\big(l_i(\Rada a{\sigma(1)}{\sigma(i)}),\Rada
a{\sigma(i+1)}{\sigma (n)}\big)=0,
\end{equation}
where $\sigma$ runs over all $(i,n-i)$-{\em unshuffles\/}
and $\varepsilon(\sigma)$ is the Koszul sign of $\sigma$. 
We thus use:

\begin{definition}
\label{sec:ainfty-algebras-bis}
An {\em \Linfty-algebra\/} is an object $\calL =
(A,l_1,l_2,l_3,\ldots)$ consisting of a graded vector space $A$ and
degree $+1$ graded symmetric linear maps $l_k : \otexp Ak\to A$, $k \geq 1$,
satisfying~(\ref{eq:8bis}) for each $n \geq 1$. 
\end{definition}

Let $\Sc A  = \bigoplus_{k \geq 1} {\mathbb S}^k A$ 
be the symmetric coalgebra with the diagonal given by the
de-concatenation; it is the cofree conilpotent cocommutative
coassociative coalgebra cogenerated by $A$.  Each coderivation 
$\omega$ of $\Sc A$ is thus determined by its components
$\omega_k := \pi \circ \omega \circ \iota_k$, $k \geq 1$, where $\pi:
\Sc A \epi A$ is the projection and $\iota_k : {\mathbb S}^k A
\hookrightarrow \Sc A$ the inclusion of the $k$th symmetric power of
$A$. We write $\omega =
(\omega_1,\omega_2,\omega_3,\ldots)$. 

Let $l := (l_1,l_2,l_3,\ldots)$ be a degree $1$ coderivation of $\Sc
A$ determined by the linear maps $l_k$ of Definition
\ref{sec:ainfty-algebras-bis}.
Axiom~(\ref{eq:8bis}) is equivalent to 
a single equation $l^2 = 0$ \cite[Theorem~2.3]{lada-markl:CommAlg95}.  So an
\Linfty-algebra is a pair $(A,l)$ a graded vector space
and a degree $+1$ coderivation $l$ of $\Sc A$ which squares to zero.

\begin{example}[Trivial \Linfty-algebra]
\label{Vcera_mi_Jarunka_volala_kdyz_jsem_byl_v_Libni-bis}
The observations of
Example~\ref{Vcera_mi_Jarunka_volala_kdyz_jsem_byl_v_Libni} apply
verbatim to the \Linfty-case --
if $\Delta$ is a degree $+1$ differential on a graded vector space
$A$, then $\calL_\Delta := (A,\Delta,0,0,\ldots)$ 
is an \Linfty-algebra.   
\end{example}

As  automorphisms of $\Tc A$ twist \Ainfty-algebras,
\Linfty-algebras can be twisted by automorphisms
$\phi : \Sc A \to \Sc A$ determined by their
components $\phi_k : {\mathbb S}^kA \to A$, $k \geq 1$. We leave as an
exercise to derive formulas for the composition
and for the twisting of $\calL_\Delta$
analogous to~(\ref{opet_jsem_podlehl}) and~(\ref{musim_napsat_zapis}).

\subsection{Explicit formulas}

For a graded associative commutative algebra with a multiplication
$\mu : \otexp A2 \to A$, denote by $\Aut(A)$ the group of
automorphisms $\phi$ of $\Sc(A)$ of the form
\[
\phi := (\id_A,f_2\, \mu_2,f_3\, \mu_3,f_4\,\mu_4,\ldots), 
\]
where $\mu_k: \otexp Ak \to A$ is the multiplication $\mu$ iterated
$(k-1)$-times, and $f_k \in \bfk$ are scalars, $k \geq 2$. To such an
automorphism we associate its {\em generating series\/} 
\begin{equation}
\label{mam-roztrhane_trenkybis}
\phi(t) = 1 + f_1t + \frac{f_2}{2!}t^2 + \frac{f_3}{3!} t^3 +\cdots \in \bfk [[t]].
\end{equation}
It is simple to verify that the generating series of the composition
of two automorphisms is the composition of their
generating series, i.e.~$(\psi\phi)(t) = \psi\big(\phi(t)\big)$.
The situation is analogous to the non-commutative case, only the
generating series involve factorials.

Let $\phi\in \Aut(A)$ be the automorphism with the generating
series~(\ref{mam-roztrhane_trenkybis}) and $\psi(t)
:= \phi^{-1}(t)$ the inverse of its generating series,
\[
\psi(t) = 1 + g_1t + \frac{g_2}{2!}t^2 + \frac{g_3}{3!} t^3
+\cdots \in \bfk [[t]].
\]
Denote by
\[
\psi'(t) := g_1 + g_2t + \frac{g_3}{2!} 
t^2 + \frac{g_4}{3!}t^3\cdots \in \bfk [[t]]
\]
its (ordinary) derivative. 
It is simple to verify  
that the components of the twisting $\Delta^\phi$ of $\Delta$ via $\phi$ are
given by   
\begin{equation}
\label{Musim_napsat_ten_uvodnik}
\Delta^\phi_k(\Rada a1k) = \sum_{\sigma} \ 
\sum_{r+s = k} c_{r} \cdot f_s \cdot  \varepsilon(\sigma) \mu_{r+1} 
\big(\Delta \mu_s(\rada{a_{\sigma(1)}}{a_{\sigma(s)}}),
\rada{a_{\sigma(s+1)}}{a_{\sigma(k)}}\big)
\end{equation}
where $\sigma$ runs over all $(r,s)$-unshuffles, $\varepsilon(\sigma)$
is the Koszul sign of $\sigma$ and 
\[
c_{r} := \frac{d^r \psi'\big(\phi(t)\big)}{dt^r}\Big|_{t=0}.
\]

\begin{example}[Classical \Linfty-braces]
\label{Dojde_zitra_Jaruska_na_parnik?}
The classical Koszul braces recalled in 
Example~\ref{dnes_vecer_s_Jarkou_a_Tondou_u_Pakousu} are
the twisting of $\calL_\Delta$ by the automorphism with the generating series
\[
\phi(t) := e^t-1 = \sum_{k \geq 1} \frac 1{k!}\ t^k  = 
t + \frac 1{2!}t^2  + \frac 1{3!}t^3 + \cdots
\]
whose inverse $\psi(t)$ equals
\[
\psi(t) = \ln(t+1) = 
\sum_{k \geq 1} \frac {(-1)^{k+1}}k\ t^k = 
t - \frac {t^2}2 + \frac {t^3}3 - \frac {t^4}4 + \cdots.
\]
Since $\psi'(t) = (1+t)^{-1}$, $\psi'\big(\phi(t)\big) = e^{-t}$,
therefore $c_{r} = (-1)^r$. As $f_s =1$ for each $s \geq 1$,
formula~(\ref{Musim_napsat_ten_uvodnik}) readily gives $\Delta^\phi_k=
\Phi^\Delta_k$ for each $k \geq 1$.
\end{example}

\begin{Example}[Hereditary exotic \Linfty-braces continued]
The braces in Example~\ref{Preletim_Vivata?} are given by the
automorphism with the generating series $\phi(t)= t + {t^2}/2$ whose
inverse equals
\[
\psi(t) = -1 + \sqrt{1+2t} = t - \frac1{2}{t^2} + \frac 1{2}t^3 -
  \frac{5  }{8}t^4 + \cdots =
t - \sum_{k\geq 2}\frac{(-t)^k}{2^{k-1}(k-1)}\binom {2k-2}k.
\] 
\end{Example}

\begin{Problem}
It is clear that an automorphism $\phi(t)$ with the generating
series~(\ref{mam-roztrhane_trenkybis}) leads to recursive braces if
and only if $f_k \in \{-1,+1\}$. 
Which property of the generating function guarantees the 
hereditarity?
\end{Problem}

\subsection{The Lie case}
\label{sec:lie-case}

One may ask how the previous material translates to the Lie algebra
case. One could expect to have, for a graded Lie algebra $L$ with a
differential~$\Delta$, natural \Cinfty-braces
$(L,\Delta,c_2,c_3,\ldots)$ emerging as the twistings of the trivial
\Cinfty-algebra by automorphisms of the Lie coalgebra $\Liec L$ and,
among these structures, a particular one that leads to higher-order
derivations of Lie algebras.

Recall that a {\em \Cinfty-algebra\/} (also called, in
\cite[\S1.4]{markl:JPAA92}, a {\em balanced \Ainfty-algebra\/}) is an
\Ainfty-algebra as in Definition~\ref{sec:ainfty-algebras}
whose structure operations vanish on decomposables of the shuffle
product. As in the \Ainfty- or \Linfty-cases,
\Cinfty-algebras  can
equivalently be described as square-zero coderivations of the
cofree conilpotent Lie coalgebra $\Liec L$ cogenerated by $L$.

We may try to proceed as in the previous two cases. 
We have the trivial \Cinfty-algebra $\calC_\Delta =
(L,\Delta,0,0,\ldots)$, thus any natural automorphism $\phi :  \Liec L
\to \Liec L$ determines a coderivation $\Delta^\phi := \phi^{-1}\Delta \phi$
that squares to $0$, i.e.~\Cinfty-braces on $L$.

The sting lies in the notion of {\em naturality\/}. In constructing
\Ainfty-braces we very crucially relied on the fact that the
cofree conilpotent coassociative coalgebra cogenerated by $A$
materialized as the tensor {\em algebra\/} $\Ten A$ equipped with the
de-concatenation diagonal. Therefore natural operations $\Ten A \to A$
give rise to natural automorphisms of $\Tc A$ and thus also to
natural \Ainfty-braces. Similarly, the cofree conilpotent coalgebra
cogenerated by $A$ can be realized as the symmetric {\em algebra\/} $\Sym V$
with the de-concatenation (unshuffle) diagonal.  

We were however not able to find an explicit and
natural (i.e.~not depending e.g.~on the choice of a basis) formula for
a diagonal on the free Lie {\em algebra\/} $\Lie L$ that would make it
a cofree conilpotent {\em co\/}algebra; we were able to describe the
diagonal for Lie worlds of length $\leq 3$ only. It is given, for
$v,v_1,v_2,v_3 \in L$, by
\begin{equation}
\label{dalsi_Jarcina_sebevrazedna_epizoda}
\begin{aligned}
D(v) &= 0
\\
D\{v_1,v_2\} &= v_1\land v_2,
\\ 
D\{v_1,\{v_2,v_3\}\}  &=2 v_1 \land \{v_2,v_3\} - v_2 \land
\{v_3,v_1\} 
- v_3 \land \{v_1,v_2\}.  
\end{aligned}
\end{equation}
In the above display, we denoted the bracket in $\Lie L$ by $\{-,-\}$
to distinguish it from the bracket of $L$ which we will denote more
traditionally by $[-,-]$.

\begin{remark}
\label{Ceka_mne_cesta_proti_strasnemu_vetru}
The lack of an explicit diagonal for the free Lie algebra $\Lie L$ may
be related to the problem of describing the Eulerian idempotents $\eu_k
: \Ten^k X \to \Ten^k X$ \cite[Corollary~1.6]{reutenauer} in terms of
iterated {\em linearly-independent\/} Lie braces. While, for
$x,x_1,x_2,x_3 \in X$,
\begin{align*}
\eu_1(x) &= x,
\\
\eu_2(x_1 \ot x_2) &= \frac1{2!} [x_1,x_2] \ \mbox { and}
\\
\eu_2(x_1 \ot x_2 \ot x_3) &=  \frac1{3!}\big([[x_1,x_2],x_3]  + 
[x_1,[x_2,x_3]]\big),
\end{align*}
a similar formula for $\eu_k$ with $k \geq 4$ is not known.
\end{remark}

Let $\Liek L$ be the subspace of $\Lie L$ spanned by
elements of the product length $k$. 
As in the previous cases, each coalgebra automorphism $\phi :
\Liec L \to \Liec L$ is determined by its components $\phi_k :
\Liek L \to L$, $k \geq 1$.
One has also
the canonical maps
$\lambda_k : \Liek L \to L$ given by the
multiplication in $L$.  Assume we found a natural
isomorphism between $\Lie L$ and $\Liec L$ such that the induced
diagonal on $\Lie L$ agrees
with~(\ref{dalsi_Jarcina_sebevrazedna_epizoda}) on elements of length
$\leq 3$. 

Let $\phi = (\id_L, f_2\lambda_2,f_3\lambda_3,\ldots)$
be the automorphism whose $k$th component equals $f_k \lambda_k$
for some scalars $f_k \in \bfk$, and $\psi = (\id_L,
g_2\lambda_2,g_3\lambda_3,\ldots)$ another one, with components $g_k
\lambda_k$, $g_k \in \bfk$.
Using~(\ref{dalsi_Jarcina_sebevrazedna_epizoda}), 
one derives the following formula for the first three components of the
composition:
\[
\psi\phi = \big(\id_L,(f_2+g_2)\lambda_2,(f_3 + 3 f_2g_2 + g_3)
\lambda_3,\ldots\big).
\]
With this formula, one easily verifies that the inverse of $\phi =
(\id_L,\lambda_2,\lambda_3,\ldots)$ is of the form  $\phi^{-1} =
(\id_L,-\lambda_2,+2\lambda_3,\ldots)$.

The twisting of the trivial \Cinfty-algebra $\calC_\Delta$ by $\phi =
(\id_L,\lambda_2,\lambda_3,\ldots)$ leads to the following formulas
for $c^\Delta_k = \Delta^\phi_k$; we for clarity omit the Koszul signs:
\begin{align*}
c^\Delta_1(v) &= \Delta(v),
\\
c^\Delta_2\{v_1,v_2\} &= 
\Delta[v_1,v_2] -  [\Delta v_1,v_2]- \NIC{v_1} [v_1,\Delta v_2],
\\
c^\Delta_3\big\{v_1,\{v_2,v_3\}\big\} &= \Delta[v_1,[v_2,v_3]]
-2\NIC{v_1}[v_1,\Delta[v_2,v_3]]+ [\Delta v_2,[v_3,v_1]] 
\\
& \hskip 1.5em
+ \NIC{v_2}[v_2,\Delta[v_3,v_1]]
+  [\Delta v_3,[v_1,v_2]] +\NIC{v_3} [v_3,\Delta[v_1,v_2]]
\\
& \hskip 1.5em
+2\big(\NIC{v_1}[v_1,[\Delta(v_2),v_3]] + \NIC{v_1+v_2}[v_1,[v_2,\Delta(v_3)]]\big).
\end{align*}
It is easy to verify that
\[
c^\Delta_3\big\{v_1,\{v_2,v_3\}\big\} = 
-2\NIC{v_1} \big[v_1,c_2^\Delta\{v_2,v_3\}\big] - 
c_2^\Delta\big\{v_2,[v_3,v_1]\big\} - 
c_2^\Delta\big\{v_3,[v_1,v_2]\big\}.
\]
Therefore, if $\Delta$ is a derivation of the Lie algebra $L$, i.e.~if
$c^\Delta_2=0$, $c^\Delta_3$ vanishes as expected.

We saw in the \Ainfty- resp.~\Linfty-cases that the twisting by the
automorphism whose components were the canonical maps $\mu_k : \Ten^k
A \to A$ resp.~$\mu_k: \Sym^k A \to A$, lead to the (unique) braces
giving a sensible notion of higher-order derivations.  This justifies:

\begin{Conjecture}
The twisting of  $\calC_\Delta =
(L,\Delta,0,0,\ldots)$ by  $\phi =
(\id_L,\lambda_2,\lambda_3,\ldots)$ gives rise to \Cinfty-braces satisfying the
analogs of conditions~(\ref{eq:12})--(\ref{eq:17}). 
\end{Conjecture}

Verifying this conjecture of course depends on describing the
isomorphism $\Lie L \cong \Liec L$.  

\subsection{Other cases}
Let us finish this note by formulating the most general context in
which our approach may work. We will
need the language of operads for which we refer for
instance to~\cite{markl:handbook,markl-shnider-stasheff:book}. Let
$\calP$ be a quadratic Koszul operad and $A$ a
$\calP$-algebra. Denote by $\calP^!$ the Koszul (quadratic) dual of
$\calP$ \cite[Def.~II.3.37]{markl-shnider-stasheff:book}
and by $\Rafik$ (resp.~$\Broucek$) the free
${\calP}$-algebra (resp.~the cofree conilpotent ${\calP}$-coalgebra)
generated (resp.~cogenerated) by $A$.

A strongly homotopy $\calP^!$-algebra, also called a
$\calP^!_\infty$-algebra,
is determined by a square-zero coderivation $p$ of $\Broucek$.  If
$\Delta$ is a differential on $A$, one 
has as before the trivial $\calP_\infty$-algebra 
${\mathcal P}_\Delta = (A,\Delta,0,0,\ldots)$ given by
extending $\Delta$ to $\Broucek$.

Under the presence of a natural identification 
$\Rafik \cong \Broucek$, one may speak about
natural automorphisms that twist $P_\Delta$ to
$\calP^!_\infty$-braces. One has the automorphism $\phi :
\Broucek \to \Broucek$ whose components are
given by the structure map $\Rafik \to A$. It is sensible
to conjecture that the related $\calP^!_\infty$-braces lead to a
reasonable notion of higher-order derivations of $\calP$-algebras.

In this general set-up, \Ainfty-braces related to associative
algebras correspond to the $\calP =  \Ass$ case, 
\Linfty-braces related to commutative associative algebras to
$\calP = \Com$, and \Cinfty-braces related to Lie algebras to 
$\calP = \Lieop$, where $\Ass$, $\Com$ and $\Lieop$ denote the
operad for associative, commutative associative and Lie algebras, 
respectively.\footnote{Recall that $\Ass^! \cong \Ass$,
  $\Com^! \cong \Lieop$ and $\Lieop^! \cong \Com$, see e.g.~\cite[Example~II.3.38]{markl-shnider-stasheff:book}.} 

\part{Naturality and acyclicity}

\section{Naturality}
\label{sec:naturality}

This section is devoted to natural operations $\otexp Ak \to A$
(resp.~$\Sym^k A \to A$), where $A$ is a graded associative
(resp.~graded commutative associative) algebra with a differential
$\Delta$.  Since in the commutative associative case the symmetric group action
brings extra complications but nothing conceptually new, we analyze in
detail only the associative  case.

\vskip .3em

\noindent 
{\bf Associative case.}  We are going introduce the space $\Nat(k)$ of
natural operations $\otexp Ak \to A$ and show that $\Nat(k)$, graded
by the degrees of maps and equipped with the differential induced by
$\Delta$, is acyclic for each $k \geq 2$. The content of this section
is a kindergarten version of the analysis of Deligne's conjecture
given in~\cite{batanin-markl}.

\subsection{Natural operations}
Intuitively, natural operations $\otexp Ak \to A$ are linear maps
composed from the data available for an arbitrary graded associative
algebra $A$ with a differential.  Equivalently, natural operations ale
linear combinations of compositions of `elementary' operations, which
are the multiplication, the differential, permutations of the inputs
and projections to the homogeneous parts.  Our categorial definition
given below is chosen so that it excludes the projections; the reason
is explained in Exercise~\ref{sec:main-results-7}. Our theory can,
however, easily be extended to include the projections as well, 
cf.~Exercise~\ref{sec:acyclicity}.  Let us start with:

\begin{example}
\label{sec:natural-operations}
The space $\Nat(1)$ of natural operations $A \to A$ is
two-dimensional, spanned by the identity $\id :A \to A \in \Nat(1)^0$
in degree $0$ and $\Delta :A \to A \in \Nat(1)^1$ in degree $1$.
The space $\Nat(2)^0$ is spanned by two operations, 
\[
a \ot b \mapsto
ab \ \mbox { and  } \ a \ot b \mapsto \zn{|a||b|}  ba,
\] 
where $ab$ resp.~$ba$ denotes the
product in $A$.  The space $\Nat(2)^1$ is $6$-dimensional, 
spanned by the operations
\begin{subequations}
\begin{equation*}
a \ot b \mapsto \Delta(a)b,\ a \ot b \mapsto \zn{|a|}   a\Delta(b),\ a \ot 
b \mapsto \Delta(ab),
\end{equation*}
and compositions of these operations with the permutation $a\ot b
\mapsto \zn{|a||b|} b \ot a$, i.e.~by
\begin{equation*}
a \ot b \mapsto \zn{|a||b|}\Delta(b)a,\ a \ot b \mapsto
\zn{|b|(|a| + 1)} b\Delta(a),\ a \ot 
b \mapsto \zn{|a||b|}\Delta(ba).
\end{equation*}
\end{subequations}
Likewise, $\Nat(2)^2$ is spanned by
\[
a \ot b \mapsto \Delta(a)\Delta(b),\ a \ot b \mapsto 
\zn{|a|}   \Delta\big(a\Delta(b)\big),\ a \ot 
b \mapsto \Delta\big(\Delta(a)b\big)
\]
and their permutations. Finally, $\Nat(2)^3$ is two-dimensional,  spanned by
\[
a\ot b \mapsto \Delta\big(\Delta(a)\Delta(b)\big) \ \mbox { and } \
a\ot b \mapsto \zn{|b|(|a| + 1)} \Delta\big(\Delta(b)\Delta(a)\big).
\]
There are no natural operations $\otexp A2 \to A$ of degrees $>4$. 
Observe that the Euler characteristic of the graded space $\Nat(2)^*$
is $2-6+6-2 = 0$. This indicates its acyclicity.
\end{example}

All operations $\beta: \otexp A2 \to A \in \Nat(2)$ listed in
Example~\ref{sec:natural-operations} share the following property. Let
$(A,\Delta_A)$ and $(B,\Delta_B)$ be graded associative 
algebras with differentials and
$\varphi : A \to B$ a linear map such that 
\begin{equation}
\label{jsem_unaveny}
\varphi(a'a'') = \varphi(a')\varphi(a'')\ \mbox { and }\
\varphi\Delta_A(a) = \Delta_B \varphi(a),\ \mbox { for each } \ a, a',a'' \in A.
\end{equation}
Then 
\begin{equation}
\label{eq:6}
\beta\big(\varphi(a')\ot \varphi(a'')\big) = \varphi\big(\beta(a'\ot
a'')\big)\  \mbox { for each } \ a',a'' \in A.
\end{equation}
Let us emphasize that we {\em do not assume\/} 
the map $\varphi$ to be homogeneous
of degree $0$, it can be an {\em arbitrary\/} linear map $A \to B$
satisfying~(\ref{jsem_unaveny}).

\begin{Example}
This example explains why we did not require the homogeneity of the
map $\varphi$ in~(\ref{jsem_unaveny}).  Consider the operation $\beta
: \otexp A2 \to A$ defined by
\[
\beta(a \ot b) := \cases{ab \in A}{if $|a| = 2$ and $|b| = -13$
  and}0{otherwise.} 
\]
It is certainly `natural' in that it is defined using the data that
are available for any graded associative algebra, 
but we do not want to consider this type
of operations.\footnote{Although, as indicated 
in Exercise~\ref{sec:acyclicity}, we
  can extend our theory to include also operations of this type.} 
What excludes $\beta$ from the family of
well-behaved operations is precisely the lack of
naturality~(\ref{eq:6}) with respect to non-homogeneous maps.

To see why it is so, take $A := \Ten(a,b)$, the tensor algebra on two
generators with $|a| := 2$ and    $|b| := -13$, and $B := \Ten(u,v)$
generated by $u$, $v$ with  $|u| = |v| := 0$. There clearly 
exists a unique `non-homogeneous' homomorphism $\varphi : A \to B$
such that $\varphi(a) = u$ and $\varphi(b) = v$.\footnote{The reason
is that $\Ten(a,b)$ is free also in the category of {\em ungraded\/}
associative algebras.} 
But then
\[
0 = \beta(u\ot v) = 
\beta\big(\varphi(a)\ot \varphi(b)\big) \not= \varphi\big(\beta(a\ot
b)\big) = \varphi(ab) = \varphi(a)\varphi(b) = uv,
\]
so $\beta$ is not natural with respect to our extended notion of a
homomorphism, though it is still natural with respect to conventional
homomorphisms as can be easily checked. 
\end{Example}

Natural operations thus appear as natural
transformations $\beta : \bigotimes^k \to \Box$ 
from the tensor power functor $\bigotimes^k : \Algsdelta \to
\Vect$ to the forgetful functor $\Box   : \Algsdelta \to
\Vect$, where $\Algsdelta$ is the category of graded associative
algebras with a differential, with morphisms as
in~(\ref{jsem_unaveny}), and $\Vect$ the category of vector
spaces. We however prefer a more explicit:

\begin{definition} 
\label{hh}
For $k \geq 1$, let $\Nat(k)$ be the abelian group of families of linear maps
\[ 
\beta_{A} : 
\otexp Ak \to A
\] 
indexed by graded associative algebras $A = (A,\Delta)$ with a differential such
that, for any linear map $\varphi:A \rightarrow B$
satisfying~(\ref{jsem_unaveny}), 
the diagram
\begin{equation}
\label{eq:16}
\xymatrix{
\otexp Ak \ar[r]^{\beta_A}\ar[d]_{\otexp \varphi k}
&
A\ar[d]^{\varphi}
\\
\otexp Bk\ar[r]^{\beta_B}
&
B
}
\end{equation}
commutes.
\end{definition}

We are going to prove a structure theorem for natural operations. 
Denote by $\Fr(x_1,\ldots,x_k)$ the free graded associative algebra
with a differential, generated by
degree $0$ elements $x_1,\ldots,x_k$ (an explicit
description is given in~\S\ref{sec:algebra-frrada-x1k-1}). Denote also by
$\Fr_{1,\ldots,1}(x_1,\ldots,x_k)$
the subspace of $\Fr(x_1,\ldots,x_k)$ spanned by the words that
contain each generator precisely once.

\begin{proposition}
\label{pojedu_Jarce_zalit_kyticky}
For each $k \geq 1$ one has a natural isomorphism
\[
\xi: \Nat(k) \cong \Fr_{1,\ldots,1}(x_1,\ldots,x_k). 
\]
\end{proposition}

\begin{proof}
Denote, for brevity, $\Fr : = \Fr(x_1,\ldots,x_k)$. 
Let $A = (A,\Delta)$ 
be an arbitrary algebra with a differential.
Given
elements $\Rada a1k \in A$, there exists a unique
$\Phi^A_{\Rada a1n} : \Fr \to A$
satisfying~(\ref{jsem_unaveny}),  specified by 
requiring $\Phi^A_{\Rada a1n}(x_i) := a_i$ for $1 \leq i
\leq k$.

Given a natural operation $\beta \in \Nat(k)$ and $\Rada a1k \in A$, 
one has the commutative diagram
\begin{equation}
\label{JarUska}
\xymatrix{
\otexp {\Fr}k \ar[r]^{\beta_{\Fr}}\ar[d]_{\otexp {(\Phi^A_{\Rada a1k})} k}
&
\Fr\ar[d]^{\Phi^A_{\Rada a1k}}
\\
\otexp Ak\ar[r]^{\beta_A}
&
A
}
\end{equation}
by naturality~(\ref{eq:16}).
We will also need the particular case of~(\ref{JarUska}) when $A = \Fr$
and $a_i = u_i x_i$, for some scalars $u_i \in \bfk$, $1 \leq i \leq k$:
\begin{equation}
\label{JarUskA}
\xymatrix{
\otexp {\Fr}k \ar[r]^{\beta_{\Fr}}
\ar[d]_{\otexp {(\Phi^A_{\rada{u_1x_1}{u_kx_k}})}k}
&
\Fr\ar[d]^{\Phi^A_{\rada{u_1x_1}{u_kx_k}}}
\\
\otexp {\Fr}k\ar[r]^{\beta_{\Fr}}
&
\ \Fr.
}
\end{equation}

Recall that $\Fr_{\rada 11} := \Fr_{\rada 11}(x_1,\ldots,x_k)$ denotes
the subspace of elements containing each generator $\Rada
x1k$ precisely once. We begin the actual proof by observing that each natural
operation $\beta \in \Nat(k)$ determines an element
$\xi(\beta) \in \Fr$ by
\begin{equation}
\label{=}
\xi(\beta) := \beta_{\Fr}(x_1 \otimes \cdots \ot x_n) \in \Fr.
\end{equation}
We will show that, quite miraculously, $\xi(\beta)$
belongs to $\Fr_{\rada 11}$.  
Clearly, $\Fr$ decomposes as
\[
\Fr = \textstyle\bigoplus_{\Rada j1k \geq 0} \Fr_{\Rada j1k}, 
\]
where $\Fr_{\Rada j1k} \subset \Fr$ is the subspace of elements
having precisely $j_i$ instances of $x_i$ for each $1 \leq i \leq
k$. The endomorphism $\Phi^A_{\rada{u_1x_1}{u_kx_k}} 
: \Fr \to \Fr$ 
acts on $\Fr_{\Rada j1k}$ 
by the multiplication with $u_1^{j_1}
\cdots u_k^{j_k}$; the subspace $\Fr_{\Rada j1k}\subset \Fr$
is, in fact, characterized by this property.  The element $\xi(\beta)$
uniquely decomposes as $\xi(\beta) = \sum_{\Rada j1k \geq 0}
\xi(\beta)_{\Rada j1k}$,
for some $\xi(\beta)_{\Rada j1k} \in \Fr_{\Rada j1k}$.

Let us turn our attention to~(\ref{JarUskA}). 
By the definition of the map $\Phi^A_{\rada{u_1x_1}{u_kx_k}}$, one has
\[
\beta_{\Fr}\big((\Phi^A_{\rada{u_1x_1}{u_kx_k}})^{\ot k}\big) 
(x_1 \ot \cdots \ot x_k)
=\beta_{\Fr}(u_1x_1 \ot \cdots \ot u_kx_k),
\]
while the linearity of $\beta_{\Fr}$ implies
\[
\beta_{\Fr}(u_1x_1 \ot \cdots \ot u_kx_k) =
u_1\cdots u_k \cdot \beta_{\Fr} (x_1 \ot \cdots \ot x_k) = u_1\cdots
u_k \cdot \xi(\beta).
\]
On the other hand
\begin{eqnarray*}
\Phi^A_{\rada{u_1x_1}{u_kx_k}}(\beta_{\Fr})(x_1 \ot \cdots \ot x_k) = 
\Phi^A_{\rada{u_1x_1}{u_kx_k}}\big(\xi(\beta)\big)
=  \hskip -.3em
\sum_{\Rada j1k \geq 0}  u_1^{j_1}\cdots u_k^{j_k}  
\cdot \xi(\beta)_{\Rada j1k},
\end{eqnarray*}
therefore the commutativity of~(\ref{JarUskA}) means that 
\[
u_1 \cdots u_k \cdot \xi(\beta) =
\sum_{\Rada j1n \geq 0}   u_1^{j_1}\cdots u_k^{j_k}  
\cdot \xi(\beta)_{\Rada j1k}
\]
for each $\Rada u1k \in \bfk$. We conclude that $\xi(\beta)_{\Rada j1k} = 0$ if
$(\Rada j1k) \not= (\rada 11)$, so $\xi(\beta) = \xi(\beta)_{\rada 11}
\in \Fr_{\rada 11}$ as claimed.

We leave as an exercise to prove prove that, vice versa, 
each element $\xi \in \Fr_{\rada 11}$ determines a 
natural operation $\beta(\xi) \in
\Nat(k)$ by the formula 
\begin{equation}
\label{Jarunka-pusa}
\beta(\xi)_{A}(a_1 \ot \cdots \ot  a_k) = 
\Phi^A_{\rada{a_1}{a_k}}(\xi),\ \mbox { for each $\Rada a1k\in A$,}
\end{equation}
cf.~the proof of~\cite[Proposition~2.9]{batanin-markl}.  The above
constructions define mutually inverse correspondences $\beta \mapsto
\xi(\beta)$ and $\xi \mapsto \beta(\xi)$ that give the isomorphism of
the proposition.
\end{proof}

\begin{Remark}
A crucial step of the previous proof was that $\xi(\beta)$
belonged to $\Fr_{\rada 11}$. It was implied by the multilinearity, which is a
particular feature of the monoidal structure given by $\ot$. In the
cartesian situation, $\xi(\beta)$ might have been an arbitrary element
of the free algebra~$\Fr(\Rada x1k)$.
\end{Remark}

\subsection{The algebra $\Fr(\Rada x1k)$.}
\label{sec:algebra-frrada-x1k-1}

Elements of the free 
algebra $\Fr = \Fr(\Rada x1k)$ are results of iterated applications of
the associative multiplication and the differential on the generators
$\Rada x1k$. The subspace $\Fr_{\rada11} = 
\Fr_{\rada11}(\Rada x1k)$ is spanned by
words containing each generator precisely once. A typical element of 
$\Fr_{\rada11}$
is thus an expression~as
\begin{equation}
\label{eq:7}
\Delta\big(\Delta(x_3)x_2\Delta(x_4)\big)x_5\Delta(x_2x_1).
\end{equation}
The algebra $\Fr$ and thus also $\Fr_{\rada11}$ is graded by the
number of occurrences of $\Delta$; the element in~(\ref{eq:7})
therefore belongs to $\Fr_{\rada11}^4$.

We can clearly encode elements of $\Fr_{\rada11}$ by
`flow diagrams' that record how the multiplication
and the differential are applied. For instance, the diagram
encoding~(\ref{eq:7}) is
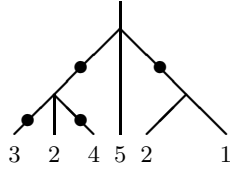
\[
{
\unitlength=.5pt
\begin{picture}(160.00,110.00)(0.00,0.00)
\thicklines
\put(160.00,0.00){\makebox(0.00,0.00)[t]{\scriptsize $1$}}
\put(100.00,0.00){\makebox(0.00,0.00)[t]{\scriptsize $2$}}
\put(80.00,0.00){\makebox(0.00,0.00)[t]{\scriptsize $5$}}
\put(60.00,0.00){\makebox(0.00,0.00)[t]{\scriptsize $4$}}
\put(30.00,0.00){\makebox(0.00,0.00)[t]{\scriptsize $2$}}
\put(0.00,0.00){\makebox(0.00,0.00)[t]{\scriptsize $3$}}
\put(50.00,20.00){\makebox(0.00,0.00){$\bullet$}}
\put(10.00,20.00){\makebox(0.00,0.00){$\bullet$}}
\put(110.00,60.00){\makebox(0.00,0.00){$\bullet$}}
\put(50.00,60.00){\makebox(0.00,0.00){$\bullet$}}
\put(100.00,10.00){\line(1,1){30.00}}
\put(80.00,90.00){\line(1,-1){80.00}}
\put(30.00,40.00){\line(0,-1){30.00}}
\put(30.00,40.00){\line(1,-1){30.00}}
\put(80.00,90.00){\line(-1,-1){80.00}}
\put(80.00,110.00){\line(0,-1){100.00}}
\put(180.00,50.00){\makebox(0.00,0.00){.}}
\end{picture}}
\] 
Its underlying graph is a rooted (= oriented) planar tree with the root
pointing upwards. The labels of its leaves (= inputs) mark the
position of the generators. The vertices symbolize iterated multiplication
while the bullets the application of the differential.

We see that elements of $\Fr_{\rada11}(\Rada x1k)$ can be represented by linear
combinations of planar rooted trees $T$ such that 
\hfill\break 
\hglue 1em -- each vertex of $T$ has at least two inputs,
\hfill\break
\hglue 1em -- all internal edges and possibly 
some external edges are decorated by the bullet $\bullet$ and
\hfill\break 
\hglue 1em -- the leaves of $T$ are labelled by a permutation of
$(\rada 1k)$.

Let us denote by $\Tr(k)$ the set of all trees as
above. Proposition~\ref{pojedu_Jarce_zalit_kyticky} together with
our description of  $\Fr_{\rada11}(\Rada x1k)$ gives

\begin{corollary}
\label{sec:algebra-frrada-x1k}
For each $k \geq 1$ one has a natural isomorphism 
$\Nat(k) \cong \Span\big(\Tr(k)\big)$.
\end{corollary}

It follows from general
theory~\cite[Proposition~II.1.27]{markl-shnider-stasheff:book} that
$\Nat(k)$ is the arity $k$-th piece of the operad $\Nat$ whose
algebras are couples $(A,\Delta)$ consisting of an associative algebra
and a differential. We will, however, not need this interpretation in
the sequel.

\begin{Example}
Corollary~\ref{sec:algebra-frrada-x1k} offers the
following description of $\Nat(2)$:
\begin{align*}
\Nat(2)^0 & =\Span\left(\rule{0pt}{18pt} \right.
\lachtanek{}{}{}12, \lachtanek{}{}{}21 \left. \rule{0pt}{18pt}  \right),
\\
\Nat(2)^1 &=\Span\left(\rule{0pt}{18pt} \right.
\lachtanek{\bullet}{}{}12,\lachtanek{}{\bullet}{}12,\lachtanek{}{}{\bullet}12,
\lachtanek{\bullet}{}{}21,
\lachtanek{}{\bullet}{}21,\lachtanek{}{}{\bullet}21
\left. \rule{0pt}{18pt}  \right),
\\
\Nat(2)^2 &=\Span\left(\rule{0pt}{18pt} \right.
\lachtanek{\bullet}{\bullet}{}12,\lachtanek{}{\bullet}{\bullet}12,
\lachtanek{\bullet}{}{\bullet}12,
\lachtanek{\bullet}{\bullet}{}21,
\lachtanek{}{\bullet}{\bullet}21,\lachtanek{\bullet}{}{\bullet}21
\left. \rule{0pt}{18pt}  \right), \ \mbox { and}
\\
\Nat(2)^3 &=\Span\left(\rule{0pt}{18pt} \right.
\lachtanek{\bullet}{\bullet}{\bullet}12,
\lachtanek{\bullet}{\bullet}{\bullet}21
\left. \rule{0pt}{18pt}  \right).
\end{align*}
We leave as an exercise to relate the above description to the
operations listed in Example~\ref{sec:natural-operations}. 
\end{Example}

\begin{example}
\label{Vcera_na_zavodech_v_Tabore_6ty}
It easily follows from Corollary~\ref{sec:algebra-frrada-x1k} that
$\Nat(k)^0$ is spanned by operations 
\[
\beta(\Rada a1k) = \varepsilon(\sigma) \cdot
a_{\sigma(1)}\cdots a_{\sigma(k)}, \ \sigma \in \Sigma_n,
\]
i.e.~$\Nat(k)^0 \cong \bfk[\Sigma_n]$ for $k \geq 1$.
\end{example}

\subsection{Acyclicity}
There is a differential on $\Nat(k)$ induced by
$\Delta$. For $\beta : \otexp Ak \to A$ it is defined by 
\[
\delta(\beta) := \Delta \beta - (-1)^{|\beta|} \sum_{1 \leq i \leq k}
\beta\big(\otexp \id {(i-1)} \ot \Delta \ot \otexp \id {(k-i)}\big) .
\]
When evaluating the above formula, we shall of course take into account the Koszul
sign convention. For instance, if $\beta \in \Nat(2)$ and $a,b,\in A$,
then 
\[
\delta(\beta)(a\ot b) =  \Delta \beta(a\ot b)- (-1)^{|\beta|}
  \beta\big(\Delta(a)\ot b\big)- (-1)^{|a| + |\beta|} 
\beta \big(a \ot
\Delta (b)\big) .
\]
It is simple to verify that $\delta^2 = 0$, so 
$\big(\Nat(k)^*,\delta\big)$ is a cochain complex. We leave as an
exercise to describe $\delta$ in terms of trees. 

\begin{proposition}
\label{Jarka_je_na_chalupe_s_M1}
The cochain complex
\begin{equation}
\label{eq:3}
\bfk = \Nat(k)^0 \stackrel \delta\longrightarrow \Nat(k)^{1} \stackrel
\delta\longrightarrow \Nat(k)^{2} \stackrel \delta\longrightarrow
\cdots
\end{equation}
is acyclic for each $k \geq 2$. 
In particular, the map $\delta : \bfk = \Nat(k)^0 \to
\Nat(k)^{1}$  is monic.
\end{proposition}

The case of $k=1$ is a particular one, as the differential $\delta :
\Nat(1)^0 \to \Nat(1)^1$ is the zero map $\Span(\id) \stackrel 0\to
\Span(\Delta)$. The explanation is that the
identity $\id : A \to A$ is the only natural operation that is
`generically' a chain map. 

\begin{proof}[Proof of Proposition~\ref{Jarka_je_na_chalupe_s_M1}] 
We describe a contracting homotopy.  By
Corollary~\ref{sec:algebra-frrada-x1k}, each natural operation $\beta
\in \Nat(k)$ is represented by a unique linear combination of trees
from $\Tr(k)$. It is therefore enough to specify how the homotopy acts
on operations given by a single tree. 

Let $\beta$ be represented by
$T \in \Tr(k)$. If the root edge of $T$ is decorated by the bullet, we
define $h(\beta)$ as the operation represented the tree $T'$ obtained
from $T$ by removing the decoration of the root.   
We define $h(\beta) := 0$ if the root edge of $T$ is not
decorated by the bullet.  We leave as an
exercise to verify that $h \delta + \delta h = \id$, so that $h$ is
a contracting homotopy.
\end{proof}

Notice that the complex $\Nat(k)$ is isomorphic to the 
direct sum of $k!$ copies of the contractible subcomplex $\uNat(k)
\subset \Nat(k)$
consisting of operations represented by trees with leaves indexed
by the identity permutation $(1,\ldots,k)$.

\begin{exercise}
\label{sec:acyclicity}
Let $\cNat(k)$ be the abelian group of families of linear maps $\beta_{A} : 
\otexp Ak \to A$ such that the diagram~(\ref{eq:16}) commutes for all
{\em homogeneous\/} $\varphi$'s.\footnote{The notation $\cNat(k)$ refers
  to \underline{c}olored natural operations.}Prove that then
\[
\cNat(k) \cong \bigoplus_{d_1,\ldots, d_k} \cNat(d_1,\ldots, d_k),
\]
where $\cNat(d_1,\ldots, d_k)$ consists of $\beta_A$'s such that 
\[
\beta_A(\Rada a1k) \not= 0\ \mbox { only if } \
(|a_1|,\ldots,|a_k|) = (d_1,\ldots, d_k).
\]
Prove, moreover, that $\cNat(d_1,\ldots, d_k) \cong \Nat(k)$ for each
$d_1,\ldots, d_k$, and discuss the acyclicity of $\cNat(k)$.
\end{exercise}

\subsection{Natural \Ainfty-algebras}
\label{sec:natur-ainfty-algebr}
Let $(A,\Delta)$ be a graded associative algebra with a
differential. We call an \Ainfty-algebra ${\mathcal A}
=(A,m_1,m_2,\ldots)$ {\em natural\/} if $m_k : \otexp Ak \to A$ are
natural operations from $\Nat(k)^1$ for each $k \geq 1$.  
Formally, a natural \Ainfty-algebra should be considered as a family
$\{{\mathcal A}_{(A,\Delta)}\}$ of \Ainfty-algebras indexed by
algebras with a differential, such that each $\varphi$ as
in~(\ref{jsem_unaveny}) induces a {\em strict\/} morphism ${\mathcal
A}_{(A,\Delta_A)} \to {\mathcal A}_{(B,\Delta_B)}$. We however believe
that our simplification will not lead to confusion.

An automorphism $\phi : \Tc A \to \Tc A$ is {\em natural\/} if all its
components $\phi_k : \otexp Ak \to A$ are natural operations from
$\Nat(k)^0$. Natural automorphisms with $\phi_1 = \id_A$ form a monoid
$\Aut(A)$.  It is clear that the twisting of a natural \Ainfty-algebra
by a natural automorphism is a natural \Ainfty-algebra.

\begin{example}
\label{snad-se-mi-neotoci-vitr}
It follows from Example~\ref{Vcera_na_zavodech_v_Tabore_6ty} that
natural automorphisms  $\phi \in \Aut(A)$ are encoded by
sequences $(\id_A,\omega_2,\omega_3,\ldots)$ of elements $\omega_k \in
\bfk[\Sigma_k]$. One has an important submonoid $\uAut \subset
\Aut(A)$ of sequences such that all $\omega_k$'s are the identities
$\id_{\Sigma_k}$. 
\end{example}

\begin{exercise}
\label{sec:natur-ainfty-algebr-1}
Assume that $\Delta$ is a derivation. Then the twisting of
$\calA_\Delta = (A,\Delta,0,0,\ldots)$ by an arbitrary natural
automorphism is {\em strictly\/} isomorphic to $\calA_\Delta$,
i.e.~$\Delta^\phi = \Delta$.
\end{exercise}

\vskip .3em

\noindent 
{\bf Commutative associative case.}
We are going to formulate commutative versions of the main statements from the
first part of this section. We omit the proofs which
are analogous to the non-commutative case. 

Let $A$ be a graded commutative associative algebra with a differential $\Delta$.
Natural operations are, analogously to
Definition~\ref{hh}, natural transformations $\beta_{A} : 
\Sym^k A \to A$ from the $k$th symmetric power of $A$ to $A$. Let us
denote by $\Nat(k)$ the abelian group of all these natural operations. To see
how $\Nat(k)$ differs from its non-commutative counterpart, we give
commutative versions of Examples~\ref{sec:natural-operations}
and~\ref{Vcera_na_zavodech_v_Tabore_6ty}.

\begin{example}[Commutative version of Example~\ref{sec:natural-operations}]
\label{Prijel_jsem-na-kole}
The space $\Nat(1)$ is, as in the non-commutative case, spanned by the
identity $\id :A \to A$ in degree $0$ and $\Delta :A \to A$ in degree~$1$.
The space $\Nat(2)^0$ is spanned by the multiplication $a \odot b \mapsto
ab$, with $\odot$ denoting the symmetric product.  
The space $\Nat(2)^1$ is two-dimensional, 
spanned by the operations
\[
a \odot b \mapsto \Delta(a)b + \zn{|a||b|}  \Delta(b)a \
\mbox { and }\
a \odot b \mapsto \Delta(ab).
\]
Likewise, $\Nat(2)^2$ is spanned by
\[
a \odot b \mapsto \Delta(a)\Delta(b) \ \mbox { and } a \odot 
b \mapsto \Delta\big(\Delta(a)b\big) - \zn{|a||b|}\Delta\big(\Delta(b)a\big).
\]
Finally, $\Nat(2)^3$ is spanned by
\[
a\odot b \mapsto \Delta\big(\Delta(a)\Delta(b)\big).
\]
The Euler characteristic of the graded space $\Nat(2)^*$
is $1-2+2-1 = 0$, so the acyclicity can be expected as in the
non-commutative case.
\end{example}

\begin{example}[Commutative version of
Example~\ref{Vcera_na_zavodech_v_Tabore_6ty}] 
\label{Vcera_na_zavodech_v_Tabore_6ty_bis}
The space
$\Nat(k)^0$ is one-dimen\-sional, spanned by the iterated multiplication
$
\mu_k(\Rada a1k) = a_1\cdots a_k,
$
so~$\Nat(k)^0 \cong \bfk$ for each $k \geq 1$.
\end{example}

An obvious modification of
Proposition~\ref{pojedu_Jarce_zalit_kyticky} holds, with
$\Fr(x_1,\ldots,x_k)$ this time the free {\em commutative\/}
associative algebra. Corollary~\ref{sec:algebra-frrada-x1k} holds as
well, with $\Tr(k)$ replaced by the space of all `abstract,'
i.e.~non-planar, trees. As in the non-commutative case, $\Delta$
induces a differential $\delta$ so that $(\Nat(k)^*,\delta)$ is
acyclic for each $k \geq 2$.

The notions of a natural \Linfty-algebras and natural automorphisms
$\phi : \Sc(A) \to \Sc(A)$ translate verbatim. The following example
however shows that the space $\Aut(A)$ of natural automorphisms is
much smaller than in the non-commutative case.

\begin{example}
\label{snad-se-mi-neotoci-vitr-bis}
The description of $\Nat(k)^0$ given in 
Example~\ref{Vcera_na_zavodech_v_Tabore_6ty_bis} implies that
natural automorphisms  $\phi \in \Aut(A)$ are encoded by
sequences $(\id_A,f_2,f_3,\ldots)$ of scalars $f_k \in
\bfk$.
\end{example}

\section{Main results}
\label{sec:main-results-5}

We are going to formulate and prove the main theorems. As
in Section~\ref{sec:naturality}, we treat in detail only the
associative non-commutative case.

\vskip .3em

\noindent 
{\bf Associative case.}  Let $A$ be a graded associative algebra with
a differential $\Delta$, and $\calA_\Delta = (A,\Delta,0,0,\ldots)$
the trivial \Ainfty-algebra of
Example~\ref{Vcera_mi_Jarunka_volala_kdyz_jsem_byl_v_Libni}. According
to the following result, each natural \Ainfty-algebra whose
linear operation $m_1$ equals $\Delta$, is uniquely given by a
twisting of $\calA_\Delta$, see \S\ref{sec:natur-ainfty-algebr} and
Definition~\ref{V_nedeli_se_Jarunka_vrati} for the meaning of
naturality and twisting. In particular, each such
an \Ainfty-algebra is
(weakly) isomorphic to $\calA_\Delta$.

\begin{subequations}
\begin{theorem}
\label{sec:main-results}
For each natural \Ainfty-algebra $\calA = (A,m_1,m_2,m_3,\ldots)$ such that
\begin{equation}
\label{eq:1}
m_1 = \Delta  
\end{equation}
there exist a {\em unique\/} natural 
automorphism $\phi = (\id_A,\phi_2,\phi_3,\ldots)$ of the
tensor coalgebra $\Tc A$ such that $\calA$ equals the twisting
of $\calA_\Delta$ via $\phi$. Explicitly
\[
m_k(a_1,\ldots,a_k) = \pi \phi^{-1}\Delta \phi(a_1,\ldots,a_k), \
\mbox { for }
\Rada ank \in A,
\]
where $\pi : \Tc A \to A$ is the canonical projection.
The \Ainfty-algebra $\calA = (A,m_1,m_2,m_3,\ldots)$ satisfies also
\begin{equation}
\label{eq:2}
m_2(a_1,a_2) = \Delta(a_1a_2) - \zn{|a_1|}
a_1\Delta(a_2) - \Delta(a_1)a_2,\  a_1,a_2 \in A,  
\end{equation}
if and only if the $\phi_2$-part of the automorphism
$\phi$ equals the product of $A$, i.e.~$\phi_2(a,b) = ab$ for each
$a,b \in A$. 
\end{theorem}
\end{subequations}

\begin{proof}
Let $m$ denote the coderivation of $\Tc A$ determined by
$(m_1,m_2,m_3,\ldots)$.  Assume that we have already constructed an
automorphism $\vartheta = (\id_A,\vartheta_2,\vartheta_3,\ldots)$ of
$\Tc A$ such that
\[
(\vartheta m \vartheta^{-1})_1= \Delta\ \mbox { and }\
(\vartheta m \vartheta^{-1})_k= 0 ,\ \mbox { for } 1 < k  \leq n,
\] 
with some $n \ge 1$. To simplify the notation, denote $n :=\vartheta
m \vartheta^{-1}$. The coderivation $n$ determines an
\Ainfty-structure of the form
$(A,\Delta,0,\ldots,0,n_{n+1},\ldots)$. 
Axiom~(\ref{eq:8}) for $n+1$ implies that $\delta(n_{n+1}) = 0$. 
Consider any automorphism $\alpha$ of
$\Tc A$ of the form $\alpha =
(\id_A,0,\ldots,0,\alpha_{n+1},\ldots)$. Clearly,
$(\alpha n\alpha^{-1})_1 = \Delta$,
$(\alpha n\alpha^{-1})_k = 0$ for $1< k \leq n$ and
\[
(\alpha n \alpha^{-1})_{n+1} = n_{n+1} +
\sum_{1 \leq j \leq n+1}\alpha_{n+1}(\id_A^{\ot(j-1)}\ot \Delta \ot
\id_A^{\ot(n-j+1)})
-\Delta \alpha_{n+1} = n_{n+1} - \delta(\alpha_{n+1}).
\]   
Since $\delta(n_{n+1}) = 0$, by
Proposition~\ref{Jarka_je_na_chalupe_s_M1} one finds
$\alpha_{n+1}$  such that $ n_{n+1} = \delta(\alpha_{n+1})$. 
With this choice, $(\alpha n\alpha^{-1})_{n+1} =
0$, thus $\phi' := \alpha\phi$ satisfies
\[
(\phi' m \phi'^{-1})_1= \Delta\ \mbox { and }\
(\phi' m \phi'^{-1})_k= 0 ,\ \mbox { for } 1 < k  \leq n+1.
\] 
This shows that we can inductively 
construct an automorphism $\phi$ of $\Tc A$ such that
$\phi m  \phi^{-1} = \Delta$ or, equivalently, $m =   \phi^{-1} \Delta
\phi$. The first part of the theorem is proven.

To demonstrate that the twisting automorphism is unique, 
assume that $m = \phi^{-1} \Delta \phi = 
\psi^{-1} \Delta \psi$. Then $\omega := \phi\psi^{-1}$ 
satisfies $\omega^{-1} \Delta \omega = \Delta$. For the bilinear
part $\omega_2$ of $\omega$ this gives
\[
\omega_2 (\Delta \ot \id_A) +\omega_2 (\id_A \ot \Delta) - \Delta \omega_2
= 0, 
\]
i.e.~$\delta(\omega_2)=0$. Since $\omega \in \Nat(k)^0$, $\omega_2 =
0$ by Proposition~\ref{Jarka_je_na_chalupe_s_M1}. In the same vein
we prove inductively that $\omega_k = 0$ for all $k \geq 2$, therefore
$\omega = \id : \Tc A \to \Tc A$ so $\phi = \psi$.
\end{proof}

Theorem~\ref{sec:main-results} combined with 
Example~\ref{snad-se-mi-neotoci-vitr} gives:

\begin{corollary}
\label{sec:main-results-6}
Natural \Ainfty-algebras satisfying~(\ref{eq:1})
are parametrized by power series
\begin{equation}
\label{eq:9}
\phi(t) := t + \omega_2t^2 + \omega_3t^3 + \omega_4 t^4 +\cdots
\end{equation}
with $\omega_k \in \bfk[\Sigma_k]$, $k \geq 2$. Natural \Ainfty-algebras
satisfying also~(\ref{eq:2}) are parametrized by
expressions~(\ref{eq:9}) with $\omega_2$  the identity permutation
$\id_{\Sigma_2}$.
\end{corollary}

A {\em weak\/} isomorphism of \Ainfty-algebras induces a {\em
strict\/} isomorphism of their cohomology algebras. We therefore get
another

\begin{Corollary}
The cohomology $H^*(A,\Delta)$ of any natural \Ainfty-algebra
$(A,\Delta,m_2,m_3,\ldots)$ is the {\em trivial\/} associative algebra with
the underlying space $H^*(A,\Delta)$.
\end{Corollary}

The observation made in Example~\ref{sec:natur-ainfty-algebr-1}
combined with Theorem~\ref{sec:main-results} lead
to another

\begin{Corollary}
Assume that $\Delta$ is a derivation of $A$. Then any natural
\Ainfty-braces vanish on $A$. 
\end{Corollary}

We finally formulate our characterization of \Boj\ braces~\cite{Boj}
recalled in Example~\ref{Jarka_mi_vcera_volala}.

\begin{theorem}
\label{sec:main-results-1}
B\"orjeson's braces $(A,\Delta,b^\Delta_2,b^\Delta_3,b^\Delta_4,\ldots)$ 
are the unique, up to a {\em strict\/}
isomorphism, natural recursive \Ainfty-braces defined over ${\mathbb Z}$ 
such that
\begin{itemize}
\item[(i)]
$b^\Delta_2$ measures the deviation of $\Delta$ from being a derivation, i.e.
\[
b^\Delta_2(a_1,a_2) = \Delta(a_1a_2) -  \Delta(a_1)a_2 -
\zn{|a_1|}a_1\Delta(a_2), \  
a_1,a_2 \in A,  
\]
\item[(ii)]
the coefficient at $\Delta(a_1a_2a_3)$ in $b^\Delta_3(a_1,a_2,a_3)$ is either $+1$ or
$-1$ and,
\item[(iii)]
the hereditarity is satisfied, that is for all $k \geq 1$,
\[
b^\Delta_k = 0 \ \mbox { implies } \ b^\Delta_{k+1} = 0.
\]
\end{itemize}
\end{theorem}

It is obvious that the B\"orjeson braces are recursive and 
satisfy~(i) and~(ii). Their
hereditarity established in~\cite{Boj} follows from an inductive
formula mentioned in Remark~\ref{sec:main-results-2}. It remains to
prove that~(i)--(iii) characterize \Boj's braces up to a strict
isomorphism. This will follow from
Propositions~\ref{sec:main-results-3} and~\ref{sec:main-results-4}
below.

\begin{proposition}
\label{sec:main-results-3}
Suppose that $A = (A,\Delta,m_2,m_3,m_4,\ldots)$ are natural
hereditary \Ainfty-braces such that $m_2 = b^\Delta_2$ and $m_3 =
b^\Delta_3$. Then $m_k = b^\Delta_k$ for any $k \geq 2$.
\end{proposition}

\begin{proof}
Assume we have already proved that
\[
b^\Delta_k = m_k \ \mbox { for } 2 \leq k \leq n,
\]
with some $n \geq 3$. The \Ainfty-axiom (\ref{eq:8}) taken with $n = k+1$
implies that $\delta(m_{n+1}) = \delta(b^\Delta_{n+1})$, i.e.~
$\delta(m_{n+1} -b^\Delta_{n+1}) = 0$. By
Proposition~\ref{Jarka_je_na_chalupe_s_M1} there exists $z_{n+1} \in
\Nat(n+1)^0$ such that 
\begin{equation}
\label{eq:14}
m_{n+1} -b^\Delta_{n+1} = \delta z_{n+1}. 
\end{equation}
Consider the free associative algebra $\Fr(\Rada x1{n+1})$
on degree $0$ variables $\Rada x1{n+1}$. Let $A_{n+1}$ be $\Fr(\Rada
x1{n+1})$ quotiented by the ideal $I_{n+1}$ generated by
\begin{equation}
\label{eq:13}
m_n(\Rada a1n) \ \mbox { for } \Rada a1n \in \Fr(\Rada x1{n+1}).
\end{equation}
Then both $b^\Delta_{n+1}$ and $m_{n+1}$ vanish on $A_{n+1}$, since both
braces are hereditary. By~(\ref{eq:14}), $\delta z_{n+1}$ must vanish
on $A_{n+1}$, too. In particular,  $\delta z_{n+1}(\Rada x1{n+1}) =
0$.\footnote{As customary, we denote both the generators of $\Fr(\Rada
  x1{n+1})$ and their equivalence classes in $A_{n+1}$ by the same symbols.} 
We are going to prove that this implies that $z_{n+1} =
0$, so  $m_{n+1}  = b^\Delta_{n+1}$ again by~(\ref{eq:14}).

It follows from the description of $\Nat(n+1)^0$ given in
Example~\ref{Vcera_na_zavodech_v_Tabore_6ty_bis} that
\[
z_{n+1}(a_1,\ldots,a_{n+1}) = \sum_{\sigma \in \Sigma_{n+1}} 
\xi_\sigma a_{\sigma(1)} a_{\sigma(1)}
\cdots a_{\sigma(n+1)}, 
\]
with some $\xi_\sigma \in \bfk$,
therefore
\begin{eqnarray}
\label{eq:11}
\lefteqn{
\delta z_{n+1}(x_1,\ldots,x_{n+1}) =} 
\\ 
\nonumber 
&&\sum_{\sigma \in \Sigma_{n+1}} 
\xi_\sigma \big(\Delta(x_{\sigma(1)}
\cdots x_{\sigma(n+1)}) - \sum_{1 \leq i \leq n+1}
x_{\sigma(1)}
\cdots \Delta(x_{\sigma(i)}) \cdots x_{\sigma(n+1)}\big).
\end{eqnarray}

The crucial observation is that modding out by the ideal $I_{n+1}$
generated by~(\ref{eq:13}) {\em does not\/} introduce any relations
involving the monomials
\begin{equation}
\label{Za_chvili_s_Jarkou_k_Pakousum.}
\Delta(x_{\sigma(1)}) 
x_{\sigma(2)}\cdots x_{\sigma(n+1)}, \ \sigma \in \Sigma_{n+1}.
\end{equation}
Let us show for instance that the degree $1$ part $I^1_{n+1}$ of the ideal
$I_{n+1}$ does not involve the element 
$\Delta(x_1)x_2 \cdots x_{n+1}$. It follows from the
definition of an ideal that the subspace of $I^1_{n+1}$ spanned by words
containing $\Rada x1{n+1}$ in this order consists of linear
combinations of the monomials
\[
x_1m_n(\Rada x2{n+1}),\ m_n(x_1x_2,\ldots, x_{n+1}),\ldots,
 m_n(x_1,\ldots, x_nx_{n+1}),\ m_n(\Rada x1{n})x_{n+1}.
\] 
Looking at the explicit form of $m_n = b^\Delta_n$ we immediately realize that
none of the above terms contains the monomial $\Delta(x_1)x_2 \cdots
x_{n+1}$. The argument for other permutations is similar.

Inspecting the coefficients at the
terms~(\ref{Za_chvili_s_Jarkou_k_Pakousum.}) in~(\ref{eq:11}), we see
that $\delta z_{n+1}=0$  in $A_{n+1}$ only if all
$\xi_\sigma$, $\sigma \in \Sigma_{n+1}$, are
trivial. So $z_{n+1} = 0$ as required, and the induction goes on.
\end{proof}

In the proof of Proposition~\ref{sec:main-results-3}, the requirement
that $n \geq 3$ in~(\ref{eq:14}) was crucial. Indeed, the monomial
$\Delta(x_1)x_2x_3$ {\em does\/} occur in $m_2(x_1,x_2)x_3$, so the argument
following formula~(\ref{Za_chvili_s_Jarkou_k_Pakousum.}) does not
work for $n=2$. We must therefore prove also the following

\begin{proposition}
\label{sec:main-results-4}
Let $\calA = (A,\Delta,m_2,m_3,m_4,\ldots)$ be a natural recursive \Ainfty-algebra
such that $m_2=b^\Delta_2$, the coefficient $C_3$ at $\Delta(a_1a_2a_3)$ in
$m_3(a_1,a_2,a_3)$ is either $1$ or $-1$, and $m_3 = 0$ implies
$m_4=0$. Then in fact $C_3 = 1$ and $m_3 = b^\Delta_3$. 
\end{proposition}

\begin{proof}
Boring tour de force. It follows from the assumptions that
the automorphism $\phi$ inducing $\calA$ must be of
the form $\phi = (\id_A,\mu,C_3\mu_3,\ldots)$. This means that $a_1,a_2,a_3$
appear in $m_3$ only in this order.  
It is simple to prove that $m_3$'s with this property form 
the family:
\begin{equation}
\label{eq:18}
\begin{aligned}
m_3(a_1,a_2,a_3) =\ & (1+\alpha) \Delta(a_1a_2a_3) - \alpha\big(
(\Delta(a_1)a_2a_3 + \sgn{a_1 + a_2} a_1a_2\Delta(a_3)\big)
\\
&+(1-\alpha) \sgn{a_1} a_1\Delta(a_2)a_3 -(\Delta(a_1a_2)a_3 + 
\sgn{a_1}a_1\Delta(a_2a_3)\big),
\end{aligned}
\end{equation}
depending on a parameter $\alpha \in \bfk$.
We did not write the signs explicitly, because they
equal the Koszul sign of the permutation of
$(\Delta,a_1,a_2,a_3)$ in the corresponding term. For
instance, $\varepsilon$ at the last term in the first line means
$|a_1| + |a_2|$.

The operation $m_4$ decomposes into the sum $m'_4 + m''_4$, where $m'_4$ is the
part that contains $a_1,a_2,a_3,a_4$ in this order.
The $m'_4$-part is parametrized by $\alpha \in \bfk$ as above
and another parameter $\beta \in \bfk$ by
\begin{align*}
m'_4(a_1,a_2,a_3,a_4) =\ & (1+\beta)\Delta(a_1a_2a_3a_4)
+(2\alpha-\beta)\big(\Delta(a_1)a_2a_3a_4 +\sgn{a_1 + a_2+a_3}  
a_1a_2a_3\Delta(a_4)\big)
\\
&+(4\alpha-\beta)\big(\sgn{a_1}  a_1\Delta(a_2)a_3a_4 + 
\sgn{a_1 + a_2}  a_1a_2\Delta(a_3)a_4\big)
\\
&- (1+\alpha)\big(\Delta(a_1a_2a_3)a_4 +
\sgn{a_1}a_1\Delta(a_2a_3a_4)\big)
\\
&- \alpha\big(\Delta(a_1a_2)a_3a_4 + \sgn{a_1 + a_2}  a_1a_2
\Delta(a_3a_4)\big)
+(1-\alpha) \sgn{a_1} a_1 \Delta(a_2a_3)a_4.  
\end{align*}
If the vanishing of $m_3$ implies the vanishing of $m_4$, then
$m'_4$ must equal the linear combination
\begin{equation}
\label{Za_chvili_s_Jaruskou_u_Mexicanu}
\begin{aligned}
A m_3(a_1a_2,&a_3,a_4) + B  m_3(a_1,a_2a_3,a_4) +C
m_3(a_1,a_2,a_3a_4)
\\
&+ Da_1m_3(a_2,a_3,a_4) + E m_3(a_1,a_2,a_3)a_4
\end{aligned}
\end{equation}
with some $A,\ldots,E \in \bfk$. Expanding $m_3$'s in the above
display gives the system:
\begin{equation}
\begin{aligned}
\label{Dnes_slavime_s_Jarkou_narozeniny.}
 (1+\beta) &= (1+\alpha)(A+B+C),
\\
(2\alpha -\beta) &= -\alpha (B+C+E)= -\alpha (A+B+D),
\\
(4\alpha - \beta) &=  (1 -\alpha)C - \alpha D  - (\alpha-1) E =
 (1 -\alpha)A - \alpha E  - (\alpha-1) D,
\\
\alpha &= \alpha A + C + E= \alpha C + A + D,
\\
(1-\alpha) &= (1-\alpha) B - D -E,\ \mbox { and }
\\
(\alpha+1) &= A+B - (\alpha+1)E = B+C - (\alpha+1)D.
\end{aligned}
\end{equation}

The case $C_3=1$ corresponds to $\alpha=0$ by~(\ref{eq:18}). The 2nd
equation of~(\ref{Dnes_slavime_s_Jarkou_narozeniny.}) immediately
gives that $\beta=0$, so the corresponding $m_2$ and $m_3$
equal the B\"orjeson braces $b^\Delta_2$ and~$b^\Delta_3$ as claimed. 

The $C_3=-1$ case happens when $\alpha = -2$. One may verify
that then the system~(\ref{Dnes_slavime_s_Jarkou_narozeniny.}) has the
{\em unique\/} solution, namely 
\[
\alpha = \beta = -2,\ A = C = 1/2, B=0 \mbox { and } D=E = -3/2.
\]
Therefore, the only way how to express $m'_4$ as a linear
combination~(\ref{Za_chvili_s_Jaruskou_u_Mexicanu}) is 
\[
\begin{aligned}
m'_4(a_1,a_2,a_3,a_4)& = 
\\
\frac12\big( m_3(a_1&a_2,a_3,a_4) +m_3(a_1,a_2,a_3a_4)\big)
-\frac32 \big(a_1m_3(a_2,a_3,a_4) + m_3(a_1,a_2,a_3)a_4\big),
\end{aligned}
\]
so $m_3 = 0$ does not imply $m_4=0$ {\em over ${\mathbb Z\/}$}. This
excludes this case.
\end{proof}

\begin{remark}
\label{sec:main-results-2}
In the $C_3=1$ case of the above proof, the solutions of the
system~(\ref{Dnes_slavime_s_Jarkou_narozeniny.}) form a 2-parametric
family depending on $A,B \in \bfk$:
\[
\alpha = \beta = 0,\ C =  1-(A+B),\ D = -A \mbox { and } E = (A+B)-1.
\]
The particular solution with $B=1$ and all other parameters trivial
gives rise to the `canonical' recursion
\[
b^\Delta_4(a_1,a_2,a_3,a_4) = b^\Delta_3(a_1,a_2a_3,a_4)
\]
that generalizes to all higher \Boj's braces in the obvious way,
proving their hereditarity.
\end{remark}

\begin{example}[Non-hereditary braces]
\label{zitra_prijede_Ronco}
Proposition~\ref{sec:main-results-3} can be used to produce examples
of non-hereditary braces. Recall from
Example~\ref{Jaruska_mozna_prijede_ve_ctvrtek_do_Prahy} 
that the \Boj\ braces are generated by
the automorphism $\phi$ with the generating series 
$\phi(t) := t + t^2 + t^3 + t^4 + \cdots$. Assume
that $\calA = (A,\Delta,m_2,m_3,\ldots)$ is the \Ainfty-algebra given by the
automorphism $\tilde\phi$ with the generating series of the
form
\[
\tilde\phi(t) := t + t^2 + t^3 + \hbox{ higher order terms.}
\] 
Clearly $m_2 = b^\Delta_2,\ m_3 = b^\Delta_3$. By
Proposition~\ref{sec:main-results-3}, $\calA$ is hereditary if and
only if $\phi(t)= \tilde\phi(t)$. So choosing e.g.~$\tilde\phi(t) := t + t^2 +
t^3$ produces non-hereditary braces.
\end{example}

\begin{example}[Super-exotic \Ainfty-braces]
\label{sec:main-results-7}
We present \Ainfty-braces that are natural in the extended sense of
Example~\ref{sec:acyclicity}. They are given, for elements
$a,a_1,a_2,\ldots$ of a 
graded associative algebra $A$ with a differential $\Delta$,~by
\begin{align*}
s^\Delta_1(a) &= \Delta(a),
\\
s^\Delta_2(a_1,a_2) &=
\ctyricases {\Delta(a_1a_2)}{if $(|a_1|,|a_2|) = (0,0)$,}
{-\Delta(a_1)a_2}{if $(|a_1|,|a_2|) = (-1,0)$,}
{-a_1\Delta(a_2)}{if $(|a_1|,|a_2|) = (0,-1)$, and}0{in the remaining
  cases,}
\\
s^\Delta_3(a_1,a_2,a_3) &= \tricases{\Delta(a_1a_2a_3)}{if
  $(|a_1|,|a_2|,|a_3|) = (0,0,0),$}{a_1\Delta(a_2)a_3}{if
  $(|a_1|,|a_2|,|a_3|) = (0,-1,0)$, and}0{in the remaining
  cases,}
\\
&\hskip .5em \vdots  
\\
s^\Delta_k(a_1,\ldots,a_k) &= \cases{\Delta(a_1\cdots a_k)}{if
  $(|a_1|,\ldots,|a_k|) = (0,\ldots,0)$, and}0{in the remaining
  cases,},\ k \geq 4. 
\end{align*}
The above braces are the \Boj\ braces of the couple $(A',\Delta)$,
where $A'$ is the associative algebra with the same underlying space
as $A$ but the multiplication $\cdot'$ defined as
\[
a \cdot' b := \cases{ab}{if $(|a|,|b|) = (0,0)$, and}{0}{in the remaining
  cases.}
\] 

Notice that the above braces are nontrivial even when $\Delta$ is a
derivation! For this reason why we did not consider this kind of
extended naturality.
\end{example}

\vskip .3em

\noindent 
{\bf Commutative associative case.}
The first part of this section translates to
the commutative case in a straightforward manner, so 
we formulate only the commutative versions of the main
theorems. We start with the commutative variant of
Theorem~\ref{sec:main-results}. Recall that $\calL_\Delta$ denotes the
trivial \Linfty-algebra from
Example~\ref{Vcera_mi_Jarunka_volala_kdyz_jsem_byl_v_Libni-bis}.

\begin{subequations}
\begin{theorem}
\label{sec:main-results_bis}
For each natural \Linfty-algebra $\calL = (A,l_1,l_2,l_3,\ldots)$ such that
\begin{equation}
\label{eq:1_bis}
l_1 = \Delta  
\end{equation}
there exist a {\em unique\/} natural 
automorphism $\phi = (\id,\phi_2,\phi_3,\ldots)$ of the
symmetric coalgebra $\Sc A$ such that $\calL$ equals the twisting
of $\calL_\Delta$ via $\phi$. 
The \Linfty-algebra $\calA = (A,l_1,l_2,l_3,\ldots)$ satisfies
\begin{equation}
\label{eq:2_bis}
l_2(a_1,a_2) = \Delta(a_1a_2) - \zn{|a_1|}   a_1\Delta(a_2) - \Delta(a_1)a_2,\  
a_1,a_2 \in A,  
\end{equation}
if and only if $\phi_2$ equals the product of $A$. 
\end{theorem}
\end{subequations}

Theorem~\ref{sec:main-results_bis} combined with the 
description of natural automorphisms  given in
Example~\ref{snad-se-mi-neotoci-vitr-bis} leads to:

\begin{corollary}
\label{sec:main-results-6_bis}
Natural \Linfty-algebras satisfying~(\ref{eq:1_bis})
are parametrized by power series
\begin{equation}
\label{eq:9_bis}
\phi(t) := t + f_2t^2 + f_3t^3 + f_4 t^4 +\cdots \in \bfk[[t]].
\end{equation}
Natural \Linfty-algebras
satisfying also~(\ref{eq:2_bis}) are parametrized by
series~(\ref{eq:9_bis}) with $f_2 = 1$.
\end{corollary}

The following theorem offers a characterization of Koszul \Linfty-braces
analogous to that of \Boj\ \Ainfty-braces given in
Theorem~\ref{sec:main-results-1}.

\begin{theorem}
\label{sec:main-results-1_bis}
The Koszul braces $(A,\Delta,\Phi^\Delta_2,\Phi^\Delta_3,\Phi^\Delta_4,\ldots)$ 
are the unique natural \Linfty-braces defined over ${\mathbb Z}$ 
such that
\begin{itemize}
\item[(i)]
$\Phi^\Delta_2$ measures the deviation of $\Delta$ from being a derivation, i.e.
\[
\Phi^\Delta_2(a_1,a_2) = \Delta(a_1a_2) -  \Delta(a_1)a_2 -
\zn{|a_1|}a_1\Delta(a_2), \  
a_1,a_2 \in A,  
\]
\item[(ii)]
the coefficient at $\Delta(a_1a_2a_3)$ in $\Phi^\Delta_3(a_1,a_2,a_3)$ is either $+1$ or
$-1$ and,
\item[(iii)]
the hereditarity is satisfied, that is for all $k \geq 1$, that is
\[
\Phi^\Delta_k = 0 \ \mbox { implies } \ \Phi^\Delta_{k+1} = 0.
\]
\end{itemize}
\end{theorem}

\begin{Exercise}
Explain how to construct non-hereditary \Linfty-braces.
\end{Exercise}

\def\cprime{$'$} \def\cprime{$'$}

\end{document}